\documentclass[a4paper, 12pt]{article}

\usepackage[sort&compress]{natbib}
\bibpunct{(}{)}{;}{a}{}{,} 

\usepackage{amsthm, amsmath, amssymb, mathrsfs, multirow, url, subfigure}
\usepackage{graphicx} 
\usepackage{ifthen} 
\usepackage{amsfonts}
\usepackage[usenames]{color}
\usepackage{fullpage}
\theoremstyle{plain} 
\newtheorem{thm}{Theorem}

\newtheorem{prop}{Proposition}
\newtheorem{lem}{Lemma}

\theoremstyle{definition}

\theoremstyle{remark}
\newtheorem{remark}{Remark}

\newtheorem*{toy}{Toy Example}

\newtheorem*{condS1}{Condition S1}
\newtheorem*{condS2}{Condition S2}

\newtheorem*{condLP1}{Condition LP1}
\newtheorem*{condGP1}{Condition GP1}
\newtheorem*{condLP2}{Condition LP2}
\newtheorem*{condGP2}{Condition GP2}

\newcommand{\RR}{\mathbb{R}}

\renewcommand{\L}{\mathcal{L}}

\newcommand{\E}{\mathsf{E}}

\newcommand{\prob}{\mathsf{P}}

\newcommand{\eps}{\varepsilon}

\newcommand{\nm}{\mathsf{N}}

\newcommand{\dir}{\mathsf{Dir}}

\newcommand{\unif}{\mathsf{Unif}}

\newcommand{\Xbar}{\bar{X}}

\renewcommand{\S}{\mathcal{S}}

\title{\Huge Data-driven priors and their posterior concentration rates}
\author{
Ryan Martin\footnote{Department of Statistics, North Carolina State University, {\tt rgmarti3@ncsu.edu}} \quad and \quad 
Stephen G.~Walker\footnote{Department of Mathematics, University of Texas at Austin, {\tt s.g.walker@math.utexas.edu}}
}
\date{\today}

\begin{document}

\maketitle 

%\doublspacing

\begin{abstract}
In high-dimensional problems, choosing a prior distribution such that the corresponding posterior has desirable practical and theoretical properties can be challenging.  This begs the question: can the data be used to help choose a good prior?  In this paper, we develop a general strategy for constructing a data-driven or {\em empirical prior} and sufficient conditions for the corresponding posterior distribution to achieve a certain concentration rate.  The idea is that the prior should put sufficient mass on parameter values for which the likelihood is large.  An interesting byproduct of this data-driven centering is that the asymptotic properties of the posterior are less sensitive to the prior shape which, in turn, allows users to work with priors of computationally convenient forms while maintaining the desired rates. General results on both adaptive and non-adaptive rates based on empirical priors are presented, along with illustrations in density estimation, nonparametric regression, and high-dimensional structured normal models.

\smallskip

\end{abstract}

\section{Introduction}
\label{S:intro}

%\subsection{Background and aims}

The Bayesian framework is ideally suited for updating prior beliefs.  However, applications often do not come equipped with genuine prior beliefs, so the data analyst must make a choice.  For low-dimensional problems, the posterior is relatively insensitive to the choice of prior, at least asymptotically, so default non-informative priors can be used.  For modern high-dimensional problems, on the other hand, the prior matters, and the present way of thinking is to choose a prior such that the corresponding posterior distribution has certain desirable properties.  For example, in sparse high-dimensional normal linear models, conjugate normal priors are attractive due to their computational simplicity.  However, it was shown in \citet[][Theorem~2.8]{castillo.vaart.2012} that, for priors with thin normal tails, the posterior has certain suboptimal asymptotic properties, so these are out and more sophisticated priors like the horseshoe \citep{carvalho.polson.scott.2010} and its variants \citep[e.g.,][]{armagan.etal.2013, dunson.shrinkage, bhadra.hsplus.2017} are now in.  The point is that, at least in high-dimensional problems, the interpretation of prior distributions has changed---their role is simply to facilitate efficient posterior inference and, therefore, only priors whose corresponding posterior has good properties are used.  So if an {\em empirical} or {\em data-dependent} prior had some practical or theoretical benefit, then there would be no reason not to use it.  This begs the two-part question: are there any benefits to the use of an empirical prior and, if so, how to construct one for which these benefits are realized?  

The idea of letting the prior depend on data is not new.  Classical empirical Bayes, as described in \citet[][Ch.~4.5]{berger1985}, \citet{carlin.louis.1996}, and more recently in \citet{efron2010book}, leaves certain prior hyperparameters unspecified and then uses the data to construct plug-in estimates of these parameters, usually via marginal maximum likelihood.  That is, if $\theta$ is the parameter of interest, then a class $\{Q_\gamma: \gamma \in \Gamma\}$ of prior distributions for $\theta$ is considered, and rather than introducing another prior for $\gamma$, one simply gets an estimator, $\hat\gamma$, based on data, and uses the plug-in prior $Q_{\hat\gamma}$.  The primary motivation for such a strategy is to let the data help carry some of the data analyst's prior specification burden.  This, in turn, can provide some computational benefits, since the posterior for $\gamma$ does not need to be evaluated.  These computational savings are usually minimal in the high-dimensional settings we have in mind here, since $\gamma$ is usually of very low dimension compared to the interest parameter $\theta$.  Posterior distribution properties for these classical empirical Bayes strategies have been investigated recently in, e.g., \citet{szabo.vaart.zanten.2013} and \citet{pas.szabo.vaart.uq, pas.szabo.vaart.rate} for a high-dimensional Gaussian model, and more generally in \citet{rousseau.etal.eb}.  These results confirm a natural conjecture that the use of the data-dependent prior $Q_{\hat\gamma}$ is asymptotically equivalent to the use of data-independent prior $Q_{\gamma^\star}$, where $\gamma^\star$ an appropriately defined ``best'' value.  But they do not reveal any theoretical benefit to the use of a data-dependent prior, it only says the performance is no worse than it would be with a special data-independent prior $Q_{\gamma^\star}$.  What is missing from the classical approach is a direct use of the information the data contains about $\theta$ itself; it only uses information indirectly through a marginal likelihood that is of little relevance to the actual problem.  

Fortunately, there are other strategies for constructing empirical priors.  \citet{martin.walker.eb} and \citet{martin.mess.walker.eb} recently employed a new type of empirical Bayes procedure, in two structured high-dimensional Gaussian linear models; related approaches to these problems can be found in \citet{belitser.ddm}, \citet{belitser.nurushev.uq}, \citet{belitser.ghosal.ebuq}, and \citet{ariascastro.lounici.2014}.  Their main idea was to suitably center the prior for $\theta$ around a good estimator, and they were able to establish various optimal posterior concentration rate and structure learning results.  An important practical consequence of their approach is that the computationally convenient conjugate normal priors, shown to be suboptimal in the classical Bayesian setting, do actually meet the conditions for optimality in this new empirical Bayes context.  The practical and theoretical benefits in these cases have been refined and extended in \citet{ebpiece}, \citet{ebcvg}, and \citet{ebpred}; see, also, \citet{lee.lee.lin.deb}.  However, their empirical prior construction and the posterior concentration results rely heavily on the Gaussian linear model structure, so whether there is a general framework underlying these developments remains an open question.  Our main contribution here is to give an affirmative answer to this question, by presenting a general empirical prior construction and establishing general posterior concentration rate results.  

To set the scene, let $X^n$ be the data, indexed by $n \geq 1$, not necessarily independent and identically distributed (iid) or even independent, with joint distribution $\prob_\theta^n$ with density $p_\theta^n$ indexed by a parameter $\theta$ in $\Theta$, possibly high- or infinite-dimensional.  For a sequence of prior distributions, $\Pi_n$, on $\Theta$, the posterior distribution, $\Pi^n$, for $\theta$ is defined, according to Bayes's formula, as 
%\begin{equation}
%\label{eq:post}
\[ \Pi^n(A) = \frac{\int_A L_n(\theta) \, \Pi_n(d\theta)}{\int_{\Theta_n} L_n(\theta) \, \Pi_n(d\theta)}, \quad A \subseteq \Theta_n, \]
%\end{equation}
where $L_n(\theta) = p_\theta^n(X^n)$ is the likelihood function.  A relevant property of the posterior $\Pi^n$ is its concentration rate relative to the Hellinger distance on the set of joint densities $\{p_\theta^n: \theta \in \Theta\}$.  Recall that the Hellinger distance between two densities, say, $f$ and $g$, with dominating measure $\mu$, is given by $H^2(f,g) = \tfrac12 \int (f^{1/2} - g^{1/2})^2 \,d\mu$.  If $\eps_n$ is a sequence with $\eps_n \to 0$ no faster than $n^{-1/2}$, then we say that the posterior distribution has (Hellinger) concentration rate (at least) $\eps_n$ at $\theta^\star$ if $\E_{\theta^\star}^n \{\Pi^n(A_{M\eps_n})\} \to 0$ as $n \to \infty$, where 
\[ A_{M\eps_n} = \bigl\{\theta: H^2(p_{\theta^\star}^n, p_\theta^n) > 1-e^{-M^2 n\eps_n^2}\bigr\} \]
and $M > 0$ is a sufficiently large constant.  Here $\E_{\theta^\star}^n$ denotes expectation with respect to the joint distribution $\prob_{\theta^\star}^n$.  For a deterministic or data-independent sequence of priors, $\Pi_n$, this property has been investigated in \citet{ggv2000} and \citet{walker2007} for the iid case and by \citet{ghosalvaart2007} in the non-iid case.  Here we investigate this property for certain data-dependent priors.  

To motivate our specific empirical prior construction, recall an essential part of the posterior concentration rate proofs for standard Bayesian posteriors.  If $\eps_n$ is the desired rate, then it is typical to consider a ``neighborhood'' of the true $\theta^\star$ of the form 
\begin{equation}
\label{realnhood}
\bigl\{\theta: K(p_{\theta^\star}^n, p_\theta^n) \leq n \eps_n^2, \, V(p_{\theta^\star}^n, p_\theta^n) \leq n \eps_n^2 \bigr\}, 
\end{equation}
where $K$ is the Kullback--Leibler divergence and $V$ is the corresponding second moment,
\[ K(f,g) = \int \log(f/g) \, f \, d\mu \quad \text{and} \quad V(f,g) = \int \log^2(f/g) \, f \, d\mu. \]
A crucial step in proving that the posterior attains the $\eps_n$ rate is to demonstrate that the prior allocates a sufficient amount of mass to the set in \eqref{realnhood}.  If the prior could be suitably centered at $\theta^\star$, then this prior concentration would be trivial.  The difficulty, of course, is that $\theta^\star$ is unknown, so care is needed to construct a prior satisfying this prior concentration property simultaneously for a sufficiently wide range of $\theta^\star$.  In fact, this placement of prior mass can be problematic and is one reason why examples like monotone density estimation are challenging; see \citet{salomond2014}. 
%\citep[e.g.,][p.~1229]{kruijer.rousseau.vaart.2010}

Our proposed alternative is motivated by considering an ``empirical version'' of the neighborhood in \eqref{realnhood}, namely, 
\[\left\{\theta: \int\log \frac{p_{\hat{\theta}_n}(x)}{p_{\theta}(x)}\, \mathbb{P}_n(dx) \leq n\,\eps_n^2\right\},\]
where $\hat{\theta}_n$ is a suitable estimator and $\mathbb{P}_n$ is the empirical distribution function. We do not need a term corresponding to the second moment, $V$. This is equivalent to 
\[ \L_n = \{\theta: L_n(\theta) \geq e^{-n \eps_n^2} L_n(\hat\theta_n)\}. \]
%where $\hat\theta_n$ is a suitable maximum likelihood estimator.  
This is effectively a neighborhood of $\hat\theta_n$, which is known, unlike the $\theta^\star$ in (\ref{realnhood}), so it is straightforward to construct a prior to assign a sufficient amount of mass to $\L_n$.  The consequence is that a prior satisfying this mass condition would depend on the data, since it must be suitably centered at $\hat\theta_n$.  But aside from the data-dependent centering and some care in its spread (see Remark~\ref{re:variance}), the specific shape of the empirical prior distribution satisfying this property is not particularly important.  Therefore, the conditions can be checked with relatively simple---often conjugate---priors, which greatly simplifies posterior computations.  Moreover, the method in general is quite versatile, providing simple solutions with optimal concentration rates in challenging problems like monotone \citep{ebmono} and heavy-tailed density estimation (Section~\ref{SS:mixture1}), and other shape-constrained problems \citep{ebpiece}, while giving improved rates in a classical nonparametric regression problem (Section~\ref{SS:np.reg}).  

%One can proceed to construct a corresponding empirical Bayes posterior by combining this empirical prior with the likelihood via Bayes's formula.  If $\hat\theta_n$ behaves sub-optimally, then the empirical prior-to-posterior update can correct for it, provided certain conditions are satisfied.  Our key observation is that an empirical prior that allocates a sufficient amount of mass to $\L_n$ is easy to arrange in practice (see Section~\ref{S:examples}) and is a significant step towards proving concentration rate results for the corresponding empirical Bayes posterior.   

%A general property of interest in the Bayesian literature is the posterior concentration rate.  To set the scene, if $\Pi$ denotes a prior distribution for some parameter $\theta \in \Theta$, with true value $\theta^\star$, then the corresponding posterior $\Pi^n$ has concentration rate $\eps_n$ if, for a constant $M > 0$,
%\begin{equation}
%\label{rates}
%\E_{\theta^\star} \bigl[\Pi^n(\{\theta: d(\theta^\star, \theta) > M \eps_n\}) \bigr] \to 0, \quad n \to \infty,
%\end{equation}
%where $d$ is a metric defined on $\Theta$.  For details on the case of independent and identically distributed (iid)~data, see \citet{ggv2000} and \citet{walker2007}; for the non-iid~case, see \citet{ghosalvaart2007}.  {\color{red} The present paper develops a new and general framework for constructing {\em empirical} or {\em data-dependent} prior distributions such that the corresponding posterior achieves the property \eqref{rates} with a target rate $\eps_n$.}

The discussion above focused on cases where the target rate $\eps_n$ was known, which can be unrealistic in high-dimensional problems.  For example, in a nonparametric regression problem, the optimal rate will depend on the smoothness of the true mean function.  If this smoothness is known, then it is possible to tune the prior so that the attainable and targeted rates agree.  However, if the smoothness is unknown, as is often the case, the prior cannot make direct use of this information, so one needs to make the prior more flexible so that it can adapt to the unknown rate.  Adaptive posterior concentration rate results have received considerable attention in the recent literature, see \citet{vaart.zanten.2009}, \citet{kruijer.rousseau.vaart.2010}, \citet{arbel.etal.sjs2013}, \citet{scricciolo2015}, and \citet{shen.ghosal.2015}.  The common feature in all this work is that the prior should be a mixture over an appropriate model complexity index.  The empirical prior approach described above can readily handle this modification, and we provide general sufficient conditions for adaptive empirical Bayes posterior concentration.  

The remainder of this paper is organized as follows.  In Section~\ref{S:eb}, we introduce the notion of an empirical prior and present the conditions needed for the corresponding posterior distribution to concentrate at the true parameter value at a particular rate.  This discussion is split into two parts, depending on whether the target rate is known or unknown.  Section~\ref{S:proofs} presents the proofs of the two main theorems, and a take-away point is that the arguments are quite straightforward, suggesting that the particular empirical prior construction is indeed very natural.  Several examples are presented in Section~\ref{S:examples}, starting from a relatively simple parametric problem and ending with a challenging adaptive nonparametric density estimation problem.  We conclude, in Section~\ref{S:discuss}, with a brief discussion.  Details for the examples are in the Appendix.

%A question that may come to mind is the following: an empirical prior already depends on the data, so why pass it through the likelihood to get a posterior?  The primary reason is that the information contained in the condition eluded to above---see the local prior conditions, LP1 and LP2, in Section~\ref{S:eb} for details---mainly pertains to the centering and scaling of the empirical prior, but the shape of the prior is important too.  In particular, capturing the dependence structure in the possibly high-dimensional parameter is essential to achieve the optimal rates, but directly specifying an empirical prior that accomplishes this is challenging.  So, passing the empirical prior through the likelihood will introduce the necessary dependence structure in the corresponding empirical Bayes posterior, opening the door for optimal concentration rates.  

\section{Empirical priors and posterior concentration}
\label{S:eb}

\subsection{Known complexity}
\label{SS:known}

For our first case, suppose the complexity of $\theta^\star$, e.g., the smoothness of the true density or regression function, is known.  Then we know the target rate, $\eps_n$, and we can make use of this information to design an appropriate sieve on which to construct an empirical prior.  For this case, below we present a set of sufficient conditions that imply the posterior corresponding to our empirical prior has Hellinger concentration rate $\eps_n$.  Applications of this result will be given in Section~\ref{S:examples}.    

Our prior construction here and in the next subsection relies on a sieve, $\Theta_n$, an increasing sequence of finite-dimensional subsets of the parameter space $\Theta$.  Let $\hat\theta_n = \arg\max_{\theta \in \Theta_n} L_n(\theta)$ be a sieve maximum likelihood estimator (MLE).  As is always the case, what distinguishes a sieve from some other subset of the parameter space is its approximation properties.  Condition~S1 below states specifically what will be required.  

\begin{condS1}
Given $\eps_n$, there exists a deterministic sequence $\theta^\dagger = \theta_n^\dagger$ in $\Theta_n$ such that 
\[ \max\bigl\{ K(p_{\theta^\star}^n, p_{\theta^\dagger}^n), V(p_{\theta^\star}^n, p_{\theta^\dagger}^n) \bigr\} \leq n \eps_n^2, \quad \text{all large $n$}. \]
\end{condS1}

\begin{remark}
\label{re:sieve}
The sequence $\theta^\dagger=\theta_n^\dagger$ in Condition~S1 can be interpreted as ``pseudo-true'' parameter values in the sense that $n^{-1} K(p_{\theta^\star}^n, p_{\theta^\dagger}^n) \to 0$.  In the case that $\Theta_n$ eventually contains $\theta^\star$, then we can trivially take $\theta^\dagger = \theta^\star$.  However, in examples like that in Section~\ref{SS:np.reg}, the model does not include the true distribution, so identifying $\theta^\dagger$ is more challenging.  Fortunately, appropriate sieves are already known in many of the key examples.  
\end{remark}

\begin{remark}
\label{re:condR1}
An important consequence of Condition~S1 is a bound on the likelihood ratio $R_n(\hat\theta_n)$ at the sieve MLE, which will be used in the proofs of our main theorems.  In particular, there exists a constant $c > 1$ such that 
\begin{equation}
\label{eq:condR1}
R_n(\hat\theta_n) \geq e^{-c n \eps_n^2} \quad \text{with $\prob_{\theta^\star}^n$-probability converging to 1}.
\end{equation}  
Indeed, for $\theta^\dagger$ in Condition~S1, by definition of $\hat\theta_n$, we trivially have $R_n(\hat\theta_n) \geq R_n(\theta^\dagger)$, and for the iid~case it follows from Lemma~8.1 in \citet{ggv2000}---with their ``$\Pi$'' a point mass at $\theta^\dagger$---that $R_n(\theta^\dagger) \geq e^{-c n \eps_n^2}$ with $\prob_{\theta^\star}^n$-probability converging to 1.  The general case is handled in Lemma~10 of \citet{ghosalvaart2007a}.
\end{remark}

The sieve $\Theta_n$ will also serve as the support of our yet-to-be-defined empirical prior $\Pi_n$.  Since it is finite-dimensional, we will assume that it is equipped with a {\em data-independent} measure $\nu_n$, e.g., Lebesgue measure, and $\Pi_n$ will have a density $\pi_n$ with respect to $\nu_n$.  The reason the measure must be data-independent is that it rules out the case of a degenerate prior supported at $\hat\theta_n$, a situation we are not interested in investigating.  

The next two conditions---LP1 and GP1---concern the prior supported on $\Theta_n$.  The first, a local prior condition, formally describes how the empirical prior $\Pi_n$ should concentrate on that empirical version of the Kullback--Leibler neighborhood \eqref{realnhood} eluded to in Section~\ref{S:intro}, namely, 
\begin{equation}
\label{eq:Ln} 
\L_n = \bigl\{\theta \in \Theta_n: L_n(\theta) \geq e^{-d n \eps_n^2} L_n(\hat\theta_n)\bigr\}, \quad \text{some $d > 0$}.  
\end{equation}
On one hand, requiring that a sufficient amount of mass be assigned to $\L_n$ is similar to the standard local prior support conditions in \citet{ggv2000}, \citet{shen.wasserman.2001}, and \citet{walker2007}, inspired by the developments in \citet{barron1988}.  On the other hand, the neighborhood's dependence on the data is our chief novelty and the main driver of our empirical prior construction.  

\begin{condLP1}
Given $\eps_n$, there exists $C > 0$ such that the prior $\Pi_n$ satisfies 
\begin{equation}
\label{eq:prior}
\prob_{\theta^\star}^n\{\Pi_n(\L_n) < e^{-Cn\eps_n^2}\} \to 0, \quad n \to \infty,
%\liminf_{n \to \infty} e^{C n \eps_n^2} \, \Pi_n(\L_n) > 0, \quad {\color{red}\text{with $\prob_{\theta^\star}^n$-probability~1}}, 
\end{equation}
where $\L_n$ is as in \eqref{eq:Ln}, depending implicitly on $\eps_n$.  
\end{condLP1}

\begin{remark}
\label{re:variance}
LP1 often requires the spread of $\Pi_n$ to be decreasing with $n$.  For example, in a scalar normal mean problem, to satisfy LP1 with $\eps_n = n^{-1/2}$ requires, say, a normal empirical prior, centered at the sample mean, with variance $v_n = v n^{-1}$ for some $v > 0$.  Of course, LP1 is a sufficient but not necessary condition, so it is possible, at least in simple cases like this, to get the desired posterior concentration rate with other priors, e.g., with constant $v_n$.  We are conditioned to believe that a tight prior is undesirable because it might be overly informative, but this rationale is based on the prior center being fixed.  In the present case, the prior gets its ``non-informativeness'' from the data-driven center.  And from this perspective, relatively tight prior concentration is actually quite reasonable, since one cannot expect a real benefit from the prior centering without putting a substantial amount of prior mass there.  Finally, when $n$ is fixed, the empirical prior spread involves constants, e.g., $v$ in $v_n$ above, that can be chosen by the data analyst, so there is no flexibility lost in practice.   
\end{remark}

%It turns out that assigning non-negligible mass to the set $\L_n$ may require that the prior be relatively tightly concentrated around $\hat\theta_n$; see Remark~\ref{re:variance}.  One might then be tempted to take this idea to the extreme and propose an empirical prior that assigns all of its mass to $\hat\theta_n$.  A degenerate prior like this would strip away all that is Bayesian about this approach, since the posterior would be the same point mass.  More importantly, the properties of such a posterior are completely determined by those of $\hat\theta_n$ but, as we demonstrate, updating a non-degenerate empirical prior, one with at least a little spread, can correct for certain sub-optimalities in $\hat\theta_n$.  {\color{red} More on this...?  The point is that we don't require, at least not directly, that the sieve MLE have a certain rate... Emphasize that our goal is, in some sense, to determine a maximal prior spread around the sieve MLE for which a posterior concentration rate result can be established, but the words ``maximal'' here will require careful explanation...}  

The second prior condition is global and effectively controls the tails of the empirical prior density $\pi_n$, i.e., how heavy can the tails be and still achieve the desired rate.  This is an empirical prior version of the more familiar prior tail condition \citep{ggv2000} or the prior summability condition \citep{walker2007} in the classical Bayesian nonparametric setting.  

\begin{condGP1}
Given $\eps_n$, there exists constants $K > 0$ and $p > 1$, such that the density function $\pi_n$ of the empirical prior $\Pi_n$ satisfies 
\[ \int_{\Theta_n} \bigl[ \E_{\theta^\star}^n\{\pi_n(\theta)^p\} \bigr]^{1/p} \, \nu_n(d\theta) \leq e^{K n \eps_n^2}. \]
%where, again, $\nu_n$ is a non-data-dependent dominating measure on $\Theta_n$.
\end{condGP1}

Condition~GP1 points to $\pi_n$ not having too heavy tails, but in a distributional sense, taking into account its dependence on the data.  While it might be unfamiliar, our examples in Section~\ref{S:examples} show that it can be verified for commonly used priors, e.g., normal priors, centered at $\hat{\theta}$, with suitable variance, and any $p > 1$.

With the empirical prior $\Pi_n$ on $\Theta_n$, having density $\pi_n$ with respect to $\nu_n$, the posterior distribution is defined as
\begin{equation}
\label{eq:post}
\Pi^n(A) = \frac{\int_A L_n(\theta) \, \pi_n(\theta) \,\nu_n(d\theta)}{\int_{\Theta_n} L_n(\theta) \, \pi_n(\theta) \, \nu_n(d\theta)}, \quad A \subseteq \Theta_n. 
\end{equation}
Then the following theorem considered the asymptotic behavior of the random variable $\Pi^n(A_{M\eps_n})$, where $A_{M\eps_n}$ is the Hellinger neighborhood described in Section~\ref{S:intro}.  While this Hellinger neighborhood is relatively specific, the result entails rates with respect to other metrics in the examples of Section~\ref{S:examples}.  And, for example, in the usual iid case, if the posterior mass assigned to $A_{M\eps_n}$ vanishes, then so does that of 
\[ \{\theta: H(p_{\theta^\star}, p_\theta) > M \eps_n\}, \]
in which case $\eps_n$ is the usual Hellinger rate.  
  
\begin{thm}
\label{thm:rate1}
Let $\eps_n$ be such that $\eps_n \to 0$ and $n \eps_n^2 \to \infty$.  If, for this $\eps_n$, Condition~S1 holds and the empirical prior satisfies LP1 and GP1, then there exists a constant $M > 0$ such that $\E_{\theta^\star}^n\{\Pi^n(A_{M\eps_n})\} \to 0$ as $n \to \infty$.  If $\eps_n = O(n^{-1/2})$, then the same conclusion holds but with the constant $M$ replace by an arbitrary sequence $M_n \to \infty$. %$\E_{\theta^\star}^n\{\Pi^n(A_{M_n\eps_n})\} \to 0$ for any $M_n \to \infty$.  
\end{thm}

\begin{proof}
See Section~\ref{S:proofs}.
\end{proof}

%{\color{red} Uniformity?  That is, can we take supremum over $\theta^\star$ in some set?  We get a deterministic upper bound on the expectation of the posterior numerator, in Lemma~2, so this is easy to make uniform.  However, in the proof of the theorem, we need certain probabilities to be $o(1)$, and I'm not sure if this can (easily) be done uniformly...}

\subsection{Unknown complexity}
\label{SS:unknown}

As discussed above, the attainable concentration rates depend on certain complexity features of the unknown $\theta^\star$, e.g., smoothness of a regression function.  If that feature is known, as in Section~\ref{SS:known}, then so is the desired rate, $\eps_n$, and that information can be used to construct a suitable sieve on which to define a prior, empirical or otherwise.  When that feature is unknown, the standard practice \citep[e.g.,][Chap.~10]{ghosal.vaart.book} is to work with a prior that mixes over models of different complexity levels, and leads to a posterior that adapts to the ``right'' complexity for the unknown $\theta^\star$.  Here we adopt that same mixture strategy, but with an empirical twist.  

Start with a representation of $\theta$ as a pair $(S, \theta_S)$, where $S$ is some model index, taking values in some finite set $\S_n$, and $\theta_S$ is the corresponding model parameter, taking values in $\Theta_{n,S}$.  This suggests a sieve 
\[ \Theta_n = \bigcup_{S \in \S_n} \Theta_{n,S}. \]
The particular form of this decomposition can vary across applications.  One that is common is to represent a log-density or regression function in terms of a basis expansion and let $\theta=(\theta_1,\theta_2,\ldots)$ denote the coefficients.  Then $S$ could correspond to a finite set of indices that are ``turned on,'' 
\[ \Theta_{n,S} = \{\theta = (\theta_1,\theta_2,\ldots): \theta_j = 0, \, j \not\in S\}, \]
and $\S_n$ a collection of subsets of $\{1,2,\ldots,\}$ whose cardinality is bounded by some specified $T_n$.  This version of $S$ is used in Sections~\ref{SS:mean} and \ref{SS:np.reg} for a sparse normal means model and nonparametric regression, respectively.  In mixture models, on the other hand, $S$ would be an integer the represents the number of mixture components.  An important feature of $S$ or of $\Theta_{n,S}$ is its dimension, which we will denote by $|S|$; in our examples, each $\Theta_{n,S}$ will be finite-dimensional and $|S|$ is literally its dimension, but this could also apply to infinite-dimensional $\Theta_{n,S}$ with $|S|$ a suitable entropy of $\Theta_{n,S}$.  The key point is that $|S|$ measures the complexity of model $S$, both intuitively and in the technical sense that a more complex $S$, one with larger $|S|$, will have a slower associated rate.  

Here, compared to Section~\ref{SS:known}, we do not know the complexity of the true $\theta^\star$ or, more specifically, we do not know which $\Theta_{n,S}$, if any, contains $\theta^\star$.  If there happens to exist a true model $S^\star$, so that $\theta^\star \in \Theta_{n,S^\star}$, then the rate we would hope to achieve is $\eps_n = \eps_{n,S^\star}$, in which case we say that the posterior concentration rate is {\em adaptive}.  More generally, if there exists a ``best'' model $S^\dagger$---see Condition~S2 below---then adaptation entails that the posterior concentrates at the associated {\em oracle rate} $\eps_n = \eps_{n,S^\dagger}$.   

The driving assumption behind recent developments in high-dimensional inference is that the truth is not too complex, and we can incorporate such a belief into our prior for $S$.  Towards this, start with a marginal prior $w_n$ for $S$, supported on $\S_n$, and a conditional prior $\Pi_{n,S}$ for $\theta_S$, given $S$, supported on $\Theta_{n,S}$.  Since $\Theta_{n,S}$ is finite-dimensional, there is some non-data-dependent measure, $\nu_{n,S}$, such as Lebesgue measure, with respect to which $\Pi_{n,S}$ has a density, $\pi_{n,S}$.  Then the prior distribution $\Pi_n$ on $\Theta_n$ is a mixture 
\begin{equation}
\label{eq:deprior1}
\Pi_n(A) = \sum_{S \in \S_n} w_n(S) \, \Pi_{n,S}(A \cap \Theta_{n,S}), \quad A \subseteq \Theta_n,  
\end{equation}
where $\Pi_{n,S}(B) = \int_B \pi_{n,S}(\theta) \, \nu_{n,S}(d\theta)$.  In practice, we often have prior information in the form of a ``low-complexity assumption,'' i.e., small $w_n(S)$ for complex $S$, but we can be non-informative about $\theta_S$, as before, by letting data control its prior center.  

As before, various conditions are needed in order to prove that the posterior concentrates at a certain rate.  Again, these come in the form of a condition on the sieve and local and global conditions on the prior.  In this case, the complexity is unknown and we seek a more general adaptive concentration result so, naturally, the conditions here are more complicated than in Section~\ref{SS:known}.  

\begin{condS2}
Given $\eps_n$, there exists $S^\dagger = S_n^\dagger$ in $\S_n$, with $|S^\dagger| \leq n\eps_n^2$, and an associated $\theta^\dagger = \theta_n^\dagger$ in $\Theta_{n,S^\dagger}$ such that 
\[ \max\bigl\{ K(p_{\theta^\star}^n, p_{\theta^\dagger}^n), V(p_{\theta^\star}^n, p_{\theta^\dagger}^n) \bigr\} \leq n \eps_n^2, \quad \text{all large $n$}. \]
\end{condS2}

Recall that the complexity of the model, as measured by $|S|$, and the quality of approximation are at odds with one another, i.e., a simple model with small $|S|$ will tend to have large Kullback--Leibler approximation error and vice versa.  The smallest $\eps_n$ for which Condition~S2 holds will be called the {\em oracle rate}.  In some examples, it is known that $\theta^\star$ belongs to $\Theta_{n,S^\star}$ for some $S^\star$, in which case we can take $S^\dagger=S^\star$ and $\theta^\dagger = \theta^\star$ so that Condition~S2 is trivial and the corresponding oracle rate is simply the rate $\eps_{n,S^\star}$ associated with the true parameter space.  In cases where $\theta^\star$ does not belong to any sieve, approximation-theoretic results are needed to check Condition~S2.  Examples of both types are presented in Section~\ref{S:examples}.  Regardless, $S^\dagger$ acts like the ``pseudo-true'' model, $\theta^\dagger$ a deterministic sequence of ``pseudo-true'' parameters, and $\eps_n = \eps_{n,S^\dagger}$ is the oracle rate; see Remark~\ref{re:sieve}.  Moreover, like in Remark~\ref{re:condR1}, Condition~S2 implies a bound on the likelihood ratio, i.e., 
\begin{equation}
\label{eq:condR2}
\prob_{\theta^\star}^n\{R_n(\hat\theta_{n, S^\dagger}) < e^{-c n \eps_n^2}\} \to 0, 
%\quad \text{with $\prob_{\theta^\star}^n$-probability $\to 1$}, 
\end{equation}
where $\hat\theta_{n,S}$ denotes the sieve MLE over $\Theta_{n,S}$, $S \in \S_n$.  

Next, similar to what we did in Section~\ref{SS:known}, let us define the sets 
\[ \L_{n,S} = \bigl\{\theta \in \Theta_{n,S}: L_n(\theta) \geq e^{-d|S|} L_n(\hat\theta_{n,S})\bigr\}, \quad S \in \S_n, \quad d > 0, \]
which are just neighborhoods of $\hat\theta_{n,S}$ in $\Theta_{n,S}$.  Then we have the following versions of the local and global prior conditions, suitable for the adaptive case, which dictate how the prior $\Pi_{n,S}$ allocates mass to $\L_{n,S}$ and $\Theta_{n,S} \cap \L_{n,S}^c$, respectively.  

\begin{condLP2}
Given $\eps_n$ and the pseudo-true model $S^\dagger$ from Condition~S2, there exist constants $A > 0$ and $C > 0$ such that, as $n \to \infty$, 
\[ \prob_{\theta^\star}^n\{ \Pi_{n,S^\dagger}(\L_{n,S^\dagger}) > e^{-C n \eps_n^2}\} \to 0, \quad \text{and} \quad w_n(S^\dagger) \gtrsim e^{-A n \eps_n^2}. \]
%\[ \liminf_{n \to \infty} e^{C n \eps_n^2} \Pi_{n,S^\dagger}(\L_{n,S^\dagger}) > 0, \quad {\color{red}\text{with $\prob_{\theta^\star}^n$-probability 1}} \]
%and   
%\[ w_n(S^\dagger) \gtrsim e^{-A n \eps_n^2}, \quad \text{large $n$}. \]
\end{condLP2}

\begin{condGP2} 
Given $\eps_n$, there exists constants $K \geq 0$ and $p > 1$ such that 
\begin{equation}
\label{eq:gp2.sum}
\sum_{S \in \S_n} w_n(S) \int_{\Theta_{n,S}} \bigl[ \E_{\theta^\star}^n \{\pi_{n,S}(\theta)^p\} \bigr]^{1/p} \, \nu_{n,S}(d\theta) \lesssim e^{K n\eps_n^2}, \quad \text{all large $n$}. 
\end{equation}
\end{condGP2}

In certain examples, such as those in Sections~\ref{SS:mean}--\ref{SS:np.reg}, it can be shown that the integral in Condition~GP2 above is bounded by $e^{\kappa |S|}$ for some constant $\kappa$.  Then the condition is satisfied with $K=0$ if the prior $w_n$ for $S$ is such that the marginal prior for $|S|$ has exponential tails \citep[e.g.,][]{arbel.etal.sjs2013, shen.ghosal.2015}.   

For adaptive concentration rates, some extra regularization is needed in addition to the prior centering.  This additional regularization amounts to a second way in which the prior depends on the data, so we refer to these as {\em double empirical priors}, and below we will consider two types of regularization.  

\begin{itemize}
\item {\em Type~1 Regularization}. For an $\alpha \in (0,1)$ to be specified, if $\Pi_n$ is the empirical prior above, then we set the double empirical prior as 
\begin{equation}
\label{eq:deprior.type1}
\widetilde \Pi_n(d\theta) \propto \frac{\Pi_n(d\theta)}{L_n(\theta)^{1-\alpha}}. 
\end{equation}
Dividing by a portion of the likelihood penalizes those parameters that ``track the data too closely'' \citep{walker.hjort.2001}, hence regularization.  A range of acceptable $\alpha$ values is identified below.  In fact, $\alpha$ can often be arbitrarily close to 1, so this is indeed a very minor adjustment.
\ifthenelse{1=1}{}{
Then the corresponding double empirical Bayes posterior has two equivalent forms:
\[ \Pi^n(d\theta) \propto L_n(\theta) \, \widetilde \Pi_n(d\theta) \quad \text{or} \quad \Pi^n(d\theta) \propto L_n(\theta)^\alpha \, \Pi_n(d\theta). \]
Of the two expressions, the former is more intuitive from an ``empirical Bayes'' perspective, while the latter resembles some recent uses of a power likelihood in, e.g., \citet{grunwald.ommen.scaling}, \citet{miller.dunson.power}, \citet{bissiri.holmes.walker.2016}, \citet{holmes.walker.scaling}, \citet{syring.martin.scaling}, and others.  
A range of acceptable $\alpha$ values is identified below.  In fact, $\alpha$ can often be arbitrarily close to 1, so this is indeed a very minor adjustment.
}  
\item {\em Type~2 Regularization}. Even though the regularization step in the above construction is very mild, some readers might be uncomfortable with what can be viewed as even a minor adjustment to the likelihood.  An alternative approach is to place the additional regularization on the prior $w_n$ for $S$.  That is, if $w_n$ is as above, then for an $\alpha \in (0,1)$ to be specified, define 
\begin{equation}
\label{eq:wtilde}
\widetilde w_n(S) \propto \frac{w_n(S)}{L_n(\hat\theta_{n,S})^{1-\alpha}}, \quad S \in \S_n. 
\end{equation}
This has the effect of putting an even smaller weight on those models that fit the data ``too well'' in the sense that their maximum likelihood is large, hence regularization.  But, as above, often any $\alpha < 1$ is allowed, so this extra regularization is quite mild.  This amounts to a double empirical prior of the form  
\begin{equation}
\label{eq:deprior.type2}
\widetilde\Pi_n(A) = \sum_{S \in \S_n} \widetilde w_n(S) \, \Pi_{n,S}(A \cap \Theta_{n,S}), \quad A \subseteq \Theta_n. 
\end{equation}
\end{itemize}

In either case, for a suitable (and implicit) $\alpha$ to be defined below, the posterior distribution based on the double empirical prior can be expressed as 
\begin{equation}
\label{eq:depost}
\Pi^n(A) = \frac{\int_A R_n(\theta) \, \widetilde\Pi_n(d\theta)}{\int_{\Theta_n} R_n(\theta) \, \widetilde\Pi_n(d\theta)}, \quad A \subseteq \Theta_n. 
\end{equation} 

\begin{thm}
\label{thm:rate2}
Let $\eps_n$ be such that $\eps_n \to 0$ and $n\eps_n^2 \to \infty$, and assume that Conditions~S2, LP2, and GP2 hold for this $\eps_n$.  For the constant $p > 1$ in Condition~GP2, take any 
\[ \alpha \in (0, 1 - p^{-1}). \]
Then there exists $M > 0$ such that $\Pi^n$ in \eqref{eq:depost}, whether it be based on Type~I or Type~II regularization, satisfies $\E_{\theta^\star}^n\{\Pi^n(A_{M\eps_n})\} \to 0$ as $n \to \infty$.   
\end{thm}

\begin{proof}
See Section~\ref{S:proofs}.
\end{proof}

We automatically have the Condition~S2 holds for any $\eps_n$ larger than the oracle rate, and since Condition~LP2 depends specifically on the pseudo-true model $S^\dagger$, it can typically be shown that it too holds for the oracle rate.  So as long as Condition~GP2 also holds for the oracle rate, we get the advertised adaptation property.  Otherwise, the rate is the larger of the oracle rate in Conditions~S2 and LP2 and that which satisfies Condition~GP2.  Moreover, if the integral in \eqref{eq:gp2.sum} is exponential in the dimension $|S|$, then Condition~GP2 can be well-controlled with weights $w_n(S)$ that are exponentially small in $|S|$.  Finally, note that Condition~GP2 can often be verified for any $p > 1$; see the examples in Section~\ref{S:examples} and the results in \citet{martin.mess.walker.eb}, \citet{ebpiece}, etc.  In such cases, any $\alpha < 1$ is allowed in either Type~I or Type~II regularization.

\section{Proofs}
\label{S:proofs}

\subsection{Proof of Theorem~\ref{thm:rate1}}

Start by expressing the posterior $\Pi^n$ in \eqref{eq:post} as 
\begin{equation}
\label{eq:posterior}
\Pi^n(A) = \frac{N_n(A)}{D_n} = \frac{\int_A R_n(\theta) \, \Pi_n(d\theta)}{\int_{\Theta_n} R_n(\theta) \, \Pi_n(d\theta)}, \quad A \subseteq \Theta_n.
\end{equation}
The dependence of the prior on data requires some modification of the usual arguments for establishing concentration properties of $\Pi^n$. In particular, in Lemma~\ref{lem:denominator1}, the lower bound on the denominator $D_n$ in \eqref{eq:posterior} is obtained quite simply thanks to the data-dependent prior, formalizing the motivation for this empirical Bayes approach described in Section~\ref{S:intro}, while Lemma~\ref{lem:numerator1} applies H\"older's inequality to get an upper bound on the numerator $N_n(A_{M\eps_n})$.    

\begin{lem}
\label{lem:denominator1}
$D_n \geq e^{-dn\eps_n^2}\,R_n(\hat\theta_n)\,\Pi_n(\L_n)$.
\end{lem}

\begin{proof}
The denominator $D_n$ can be trivially lower-bounded as follows:
\[ D_n \geq \int_{\L_n} R_n(\theta) \, \pi_n(\theta) \, \nu_n(d\theta) = R_n(\hat\theta_n) \int_{\L_n} \frac{L_n(\theta)}{L_n(\hat\theta_n)} \, \pi_n(\theta) \, \nu_n(d\theta).  \]
Now use the definition of $\L_n$ to complete the proof.
\end{proof}

\begin{lem}
\label{lem:numerator1}
Assume Condition~GP1 holds for $\eps_n$ with constants $(K,p)$, and let $q > 1$ be the H\"older conjugate of $p$.  Then 
\[ \E_{\theta^\star}^n \Bigl\{ \frac{N_n(A_{M\eps_n})}{R_n(\hat\theta_n)^{1-\frac{1}{2q}}} \Bigr\} \leq e^{-G n \eps_n^2}, \]
where $G = M^2 q^{-1} - K$.  
\end{lem}

\begin{proof}
Start with the following simple bound:
\begin{align*}
N_n(A_{M\eps_n}) & = \int_{A_{M\eps_n}} R_n(\theta) \pi_n(\theta) \,\nu_n(d\theta) \\
& \leq R_n(\hat\theta_n)^{1-\frac{1}{2q}} \int_{A_{M\eps_n}} R_n(\theta)^{\frac{1}{2q}} \pi_n(\theta) \,\nu_n(d\theta). 
\end{align*}
Dividing both sides by $R_n(\hat\theta_n)^{1-\frac{1}{2q}}$, and taking expectations,  moving this expectation inside the integral, and applying H\"older's inequality, gives 
\[ \E_{\theta^\star}^n \Bigl\{ \frac{N_n(A_{M\eps_n})}{R_n(\hat\theta_n)^{1-\frac{1}{2q}}} \Bigr\} \leq \int_{A_{M\eps_n}} \bigl[\E_{\theta^\star}^n \{R_n(\theta)^{\frac12}\} \bigr]^{\frac1q} \bigl[ \E_{\theta^\star}^n \{\pi_n(\theta)^p\} \bigr]^{\frac1p} \,\nu_n(d\theta). \]
A standard argument \citep[e.g.,][]{walker.hjort.2001} shows that the first expectation on the right hand side above equals $1-H^2(p_{\theta^\star}^n, p_\theta^n)$ and, therefore, is upper bounded by $e^{-M^2 n \eps_n^2}$, uniformly in $\theta \in A_{M\eps_n}$.  Under Condition~GP1, the integral of the second expectation is bounded by $e^{Kn\eps_n^2}$.  Combining these two bounds proves the claim.  
\end{proof}

\begin{proof}[Proof of Theorem~\ref{thm:rate1}]
To start, set 
\[ a_n = e^{-c n \eps_n^2} \quad \text{and} \quad b_n =c_0\, e^{-(C+d)n\eps_n^2}\,R_n(\hat\theta_n), \]
where the constants $(C, c, d)$ are as in Condition~LP1, Remark~\ref{re:condR1}, and Equation \eqref{eq:Ln}, respectively, and $c_0$ is another sufficiently small constant.  Also, abbreviate $N_n = N_n(A_{M\eps_n})$ and $R_n = R_n(\hat\theta_n)$.  If $1(\cdot)$ denotes the indicator function, then 
\begin{align*}
\Pi^n(A_{M\eps_n}) & = \frac{N_n}{D_n} 1(R_n \geq a_n \text{ and } D_n \geq b_n) + \frac{N_n}{D_n} 1(R_n < a_n \text{ or } D_n < b_n) \\
& \leq \frac{R_n^{1-\frac{1}{2q}}}{b_n} \, \frac{N_n}{R_n^{1-\frac{1}{2q}}}1(R_n \geq a_n) + 1(R_n < a_n) + 1(D_n < b_n) \\
& \leq \frac{e^{(C+d)n\eps_n^2}}{a_n^{\frac{1}{2q}}} \, \frac{N_n}{R_n^{1-\frac{1}{2q}}} + 1(R_n < a_n) + 1(D_n < b_n) \\
& = e^{(C + \frac{c}{2q} + d) n \eps_n^2} \, \frac{N_n}{R_n^{1-\frac{1}{2q}}} + 1(R_n < a_n) + 1(D_n<b_n), 
\end{align*}
Taking expectation and applying Lemma~\ref{lem:numerator1}, we get 
\begin{equation}
\label{eq:finite.n.bound}
\E_{\theta^\star}^n \{\Pi^n(A_{M\eps_n})\} \leq e^{(C + \frac{c}{2q} + d) n \eps_n^2} e^{-G n \eps_n^2} + \prob_{\theta^\star}^n(R_n < a_n) + \prob_{\theta^\star}^n(D_n < b_n). 
\end{equation}
The second and third terms are $o(1)$ by Remark~\ref{re:condR1} and Lemma~\ref{lem:denominator1}, respectively.  If we take $G > C + \frac{c}{2q} + d$ or, equivalently, $M^2 > q(K + C + \frac{c}{2q} + d)$, then the first term is $o(1)$ as well, completing the proof of the first claim.

For the second claim, when $n \eps_n^2$ is bounded, the conclusion \eqref{eq:finite.n.bound} still holds, and the latter two terms are still $o(1)$.  The first term in the upper bound is decreasing in $G$ or, equivalently, in $M$, so the upper bound vanishes for any $M_n \to \infty$.  
%for large $n$ we can make the upper bound in \eqref{eq:finite.n.bound} arbitrarily small by choosing $M$ sufficiently large.  
\end{proof}

\subsection{Proof of Theorem~\ref{thm:rate2}}

The proof approach here is similar to that of Theorem~\ref{thm:rate1} above, with a few differences.  We will start with the posterior defined by the double empirical prior with Type~1 regularization described in Section~\ref{SS:unknown}.  For that version of the prior, the the posterior probability $\Pi^n(A_{M\eps_n})$ is a ratio $N_n(A_{M\eps_n})/ D_n$, where 
\[ N_n(A_{M\eps_n}) = \sum_{S \in \S_n} w_n(S) \int_{A_{M\eps_n} \cap \Theta_{n,S}} R_n(\theta)^\alpha \, \pi_{n,S}(\theta) \,\nu_{n,S}(d\theta) \]
and 
\[ D_n = \sum_{S \in \S_n} w_n(S) \int_{\Theta_{n,S}} R_n(\theta)^\alpha \,\pi_{n,S}(\theta) \,\nu_{n,S}(d\theta). \]
After proving Theorem~\ref{thm:rate2} for this case, we will describe the adjustments needed to get the same result with the Type~2 regularized double empirical prior.  Throughout, we will assume Conditions~S2, LP2, and GP2 hold with $\eps_n$.  

\begin{lem}
\label{lem:denominator2}
$D_n \geq e^{-d|S^\dagger|} \, w_n(S^\dagger) \, R_n(\hat\theta_{n,S^\dagger})^\alpha \, \Pi_{n,S^\dagger}(\L_{n,S^\dagger})$.
\end{lem}

\begin{proof}
Almost identical to the proof of Lemma~\ref{lem:denominator1}.  
\end{proof}

\begin{lem}
\label{lem:numerator2}
Let $K \geq 0$ and $p > 1$ be the constants in Condition~GP2, let $q > 1$ be the H\"older conjugate of $p$, and take $\alpha$ in $(0,1-p^{-1})$.  Then 
\[ \E_{\theta^\star}^n \{ N_n(A_{M\eps_n}) \} \lesssim e^{-G n\eps_n^2}, \]  
where $G=M^2 k - K$ and $k$ depends only on $\alpha$ and $q$.
\end{lem}

\begin{proof}
Abbreviate $N_n(A_{M\eps_n})$ by $N_n$.  Taking expectation of $N_n$, moving expectation inside the integral, and applying H\"older's inequality, we get 
\[ \E_{\theta^\star}^n (N_n) \leq \sum_{S \in \S_n} w_n(S) \int_{A_{M\eps_n} \cap \Theta_{n,S}} \bigl[\E_{\theta^\star}^n \{R_n(\theta)^{\alpha q}\} \bigr]^{\frac1q} \bigl[ \E_{\theta^\star}^n \{\pi_{n,S}(\theta)^p\} \bigr]^{\frac1p} \,\nu_{n,S}(d\theta). \]
Consider the first expectation on the right-hand side, $\E_{\theta^\star}^n\{R_n(\theta)^{\alpha q}\}$.  Let $r=\alpha q$, which is in $(0,1)$ by the choice of $\alpha$.  Then the expected likelihood ratio can be written as $\int f^r g^{1-r} \, d\mu$, where $f$ and $g$ are joint densities corresponding to $\theta$ and $\theta^\star$, respectively.  This latter integral is related to the R\'enyi divergence of order $r$ which, in turn, is related to the Hellinger distance.  Indeed, by Theorem~16 in \citet{vanerven.renyi}, it is easy to see that 
\[ \E_{\theta^\star}^n\{R_n(\theta)^{\alpha q}\} \leq \{1 - H^2(p_{\theta^\star}^n, p_\theta^n)\}^{k'}, \]
where $k'$ only depends on $\alpha q$.  Therefore, by definition of $A_{M\eps_n}$, the right-hand side is upper bounded by $e^{-M^2 k' n \eps_n^2}$, uniformly in $\theta \in A_{M\eps_n} \cap \Theta_{n,S}$ and in $S$, so 
\[ \E_{\theta^\star}^n (N_n) \leq e^{-(M^2k'/q) n \eps_n^2} \sum_{S \in \S_n} w_n(S) \int_{A_{M\eps_n} \cap \Theta_{n,S}} \bigl[ \E_{\theta^\star}^n \{\pi_{n,S}(\theta)^p\} \bigr]^{\frac1p} \,\nu_{n,S}(d\theta). \]
Under Condition~GP2, the summation on the right-hand side above is bounded by a constant times $e^{K n\eps_n^2}$ and the claim now follows with $k=k' q^{-1}$.
\end{proof}

\begin{proof}[Proof of Theorem~\ref{thm:rate2}]
By Lemma~\ref{lem:denominator2} and Condition~LP2, 
\[ D_n \geq e^{-d|S^\dagger|} e^{-A n\eps_n^2} R_n(\hat\theta_{n, S_n^\star})^\alpha e^{-C n \eps_n^2}. \]
And by \eqref{eq:condR2} we have $R_n(\hat\theta_{n,S^\dagger}) \geq e^{-c n \eps_n^2}$ for some $c > 1$, with $\prob_{\theta^\star}^n$-probability converging to 1.  Since $|S^\dagger| \leq n\eps_n^2$, this lower bound for the denominator can be combined with the upper bound in the numerator from Lemma~\ref{lem:numerator2} using an argument very similar to that in the proof of Theorem~\ref{thm:rate1}, to get 
\[ \E_{\theta^\star}^n\{\Pi^n(A_{M\eps_n})\} \leq e^{-\{M^2 k - (K + A + C + c\alpha + d)\} n \eps_n^2} + o(1). \]
So, for $M$ sufficiently large, the upper bound vanishes, proving the claim.
\end{proof}

It turns out that the proof for the Type~2 regularized double empirical prior follows along almost the same lines.  The key is that we do not need to be concerned about the normalizing constant in the definition of $\widetilde w_n$ in \eqref{eq:wtilde} because it appears in both the numerator and denominator of the posterior probability.  Similarly, we can replace $L_n(\hat\theta_{n,S})$ in \eqref{eq:wtilde} by $R_n(\hat\theta_{n,S})$ so, for the proof, we are free to assume that
\[ \widetilde w_n(S) = \frac{w_n(S)}{R_n(\hat\theta_{n,S})^{1-\alpha}}, \quad S \in \S_n. \]
With this, the bound on the denominator, $D_n$, of the posterior probability from Lemma~\ref{lem:denominator2} is unchanged.  For the numerator, $N_n(A_{M\eps_n})$, note the following trivial inequality: 
\[ R_n(\theta) = R_n(\theta)^{1-\alpha} R_n(\theta)^\alpha \leq R_n(\hat\theta_{n,S})^{1-\alpha} R_n(\theta)^\alpha. \]
Consequently, 
\[ \widetilde w_n(S) \int_{A_{M\eps_n} \cap \Theta_{n,S}} R_n(\theta) \, \pi_{n,S}(\theta) \, \nu_{n,S} \leq w_n(S) \int_{A_{M\eps_n} \cap \Theta_{n,S}} R_n(\theta)^\alpha \, \pi_{n,S}(\theta) \, \nu_{n,S}, \]
and the right-hand side is exactly what we bounded in the proof of Lemma~\ref{lem:numerator2}.  So we can put together the bounds on the numerator and denominator exactly like we did above to obtain the $\eps_n$ posterior convergence rate for the Type~2 regularized version of the double empirical prior.

\section{Examples}
\label{S:examples}

\subsection{Fixed finite-dimensional parameter estimation}
\label{SS:finite.dim}

Suppose that the parameter space $\Theta$ is a fixed subset of $\RR^d$, for a fixed $d < \infty$.  Under the usual regularity conditions, the log-likelihood $\ell_n = \log L_n$ is twice continuously differentiable, its derivative $\dot\ell_n$ satisfies $\dot\ell_n(\hat\theta_n) = 0$ at the (unique) global MLE $\hat\theta_n$, and the following expansion holds:
\begin{equation}
\label{eq:quadratic}
\ell_n(\theta) - \ell_n(\hat\theta_n) = -\tfrac12 (\theta - \hat\theta_n)^\top \hat{\Sigma}_n (\theta - \hat\theta_n) + o(n\|\theta-\hat\theta_n\|^2), 
\end{equation}
where $\hat{\Sigma} _n = -\ddot\ell_n(\hat\theta_n)$.  Then the set $\L_n$ can be expressed as 
\[ \L_n = \bigl\{\theta: (\theta - \hat\theta_n)^\top \Sigma_n (\theta - \hat\theta_n) < a n\eps_n^2 \bigr\}. \]
For rate $\eps_n=n^{-1/2}$, this suggests an empirical prior of the form:
\begin{equation}
\label{eq:finite.dim.prior}
\Pi_n = \nm_d\bigl(\hat\theta_n, n^{-1}\Psi^{-1}\bigr), 
\end{equation}
for some fixed positive definite matrix $\Psi$ in order to ensure S1. The proposition below states that this empirical prior yields a posterior that concentrates at the parametric rate $\eps_n = n^{-1/2}$.  Note that we do not need any additional fine-tuning, like in Theorem~2.4 of \citet{ggv2000}, to get optimal rates in the finite-dimensional case. 

\begin{prop}
\label{prop:finite.dim}
Assume that each component $\theta_j$ in the $d$-dimensional parameter $\theta$ are on $(-\infty, \infty)$, and that the regularity conditions necessary to establish the quadratic approximation \eqref{eq:quadratic} hold.  Then Conditions~LP1 and GP1 hold for the empirical prior \eqref{eq:finite.dim.prior} with $\eps_n = n^{-1/2}$.  Therefore, the posterior, with $\alpha=1$, concentrates at the rate $\eps_n=n^{-1/2}$ relative to any metric on $\Theta$.  
\end{prop}

\begin{proof}
See the Appendix.  
\end{proof}  

%{\color{red} Any relevant comments here?  For example, is it interesting that no modifications of the general nonparametric-ready method/proof are needed to get exactly the root-$n$ rate in the finite-dimensional case?  At least in some cases, e.g., the nonparametric regression example, it seems that our approach gives the optimal rate, not a logarithmic multiple thereof, but I'm not sure if this is a general phenomenon or just a coincidence... }

\subsection{Density estimation via histograms}
\label{SS:histogram}

Consider estimation of a density function, $p$, supported on the compact interval $[0,1]$, based on iid samples $X_1,\ldots,X_n$.  A simple approach to develop a Bayesian model for this problem is a random histogram prior \citep[e.g.,][]{scricciolo2007, scricciolo2015}.  That is, we consider a partition of the interval $[0,1]$ into $S$ bins of equal length, i.e., $[0,1] = \bigcup_{s=1}^S E_s$, where $E_s = [\frac{s-1}{S}, \frac{s}{S})$, $s=1,\ldots,S$.  For a given $S$, write the model
\[ p_\theta(x) = \sum_{s=1}^S \theta_s \, \unif(x \mid E_s), \quad x \in [0,1], \]
consisting of mixtures of uniforms, i.e., piecewise constant densities, where the parameter $\theta$ is a vector in the $S$-dimensional probability simplex, 
\[ \Delta(S) = \{(\theta_1,\ldots,\theta_S): \theta_s \geq 0, \, \textstyle\sum_{s=1}^S \theta_s = 1\}. \]
That is, $p_\theta$ is effectively a histogram with $S$ bins, all of the same width, $S^{-1}$, and the height of the $s^\text{th}$ bar is $S^{-1} \theta_s$, $s=1,\ldots,S$.  Here, assuming the regularity of the true density is known, we construct an empirical prior for the vector parameter $\theta$ such that, under conditions on the true density, the corresponding posterior on the space of densities has Hellinger concentration rate within a logarithmic factor of the minimax rate.  More sophisticated models for density estimation will be presented in Sections~\ref{SS:mixture1} and \ref{SS:mixture2}. 

Let $S=S_n$ be the number of bins, specified below.  This defines a sieve $\Theta_n = \Delta(S_n)$ and, under the proposed histogram model, the data can be treated as multinomial, so the (sieve) MLE is $\hat\theta_n = (\hat\theta_{n,1},\ldots,\hat\theta_{n,S})$, where $\hat\theta_{n,s}$ is just the proportion of observations in the $s^\text{th}$ bin, $s=1,\ldots,S$.  Here we propose a Dirichlet prior $\Pi_n$ for $\theta$, namely, 
\[ \theta \sim \Pi_n = \dir_S(\hat\alpha), \quad \hat\alpha_s = 1 + c \, \hat\theta_{n,s}, \quad s=1,\ldots,S, \]
which is centered on the sieve MLE in the sense that the mode of the empirical prior density is $\hat\theta_n$; the factor $c=c_n$ will be specified below.  Finally, this empirical prior for $\theta$ determines an empirical prior for the density via the mapping $\theta \mapsto p_\theta$.  

\begin{prop}
\label{prop:histogram}
Suppose that the true density, $p^\star$, is uniformly bounded away from 0 and is H\"older continuous with smoothness parameter $\beta$, where $\beta \in (0,1]$ is assumed to be known.  Set $\eps_n = n^{-\kappa} \log^\kappa n$, where $\kappa = \beta / (2\beta + 1)$.  For the empirical prior $\Pi_n$ described above, if $S=S_n=n \eps_n^2 (\log n)^{-1}$ and $c=c_n = n \eps_n^{-2}$, then there exists $M > 0$ such that the corresponding posterior $\Pi^n$, with $\alpha=1$, satisfies 
\[ \E_{p^\star}^n\bigl[ \Pi^n(\{\theta: H(p^\star, p_\theta) > M \eps_n\}) \bigr] \to 0. \]
\end{prop} 

\begin{proof}
See the Appendix.  
\end{proof}

\subsection{Mixture density estimation}
\label{SS:mixture1}

Let $X_1,\ldots,X_n$ be iid~samples from a density $p_\theta$ of the form
\begin{equation}
\label{eq:mixture}
p_\theta(x)=\int k(x \mid \mu) \, \theta(d\mu), 
\end{equation}
where $k(x \mid \mu)$ is a known kernel and the mixing distribution $\theta$ is unknown.  Here we focus on the normal mixture case, where $k(x \mid \mu) = \nm(x \mid \mu, \sigma^2)$, where $\sigma$ is known, but see Remark~\ref{re:kernel}.  The full parameter space $\Theta$, which contains the true mixing distribution $\theta^\star$, is the set of all probability measures on the $\mu$-space, but we consider here a finite mixture model of the form 
\begin{equation}
\label{eq:finite.mixture}
\theta = (\omega,\mu) \mapsto p_\theta(\cdot) = \sum_{s=1}^S \omega_s \, k(\cdot \mid \mu_s), 
\end{equation}
for an integer $S$, a vector $\omega=(\omega_1,\ldots,\omega_S)$ in the simplex $\Delta(S)$, and a set of distinct support points $\mu=(\mu_1,\ldots,\mu_S)$.  For fixed $S$, let $\hat\theta=(\hat \omega, \hat \mu)$ be the MLE for the mixture weights and locations, respectively, where the optimization is restricted so that $|\hat \mu_s| \leq B$, where $B=B_n$ is to be determined.  We propose to ``center'' an empirical prior on the $S$-specific MLE as follows:
\begin{itemize}
\item $\omega$ and $\mu$ are independent; 
\vspace{-2mm}
\item the vector $\omega$ is $\dir_S(\hat\alpha)$ like in Section~\ref{SS:histogram}, where $\hat\alpha_s = 1 + c \, \hat \omega_s$, $s=1,\ldots,S$;  
\vspace{-2mm}
\item the components $(\mu_1,\ldots,\mu_S)$ of $\mu$ are independent, with 
\[ \mu_s \sim \unif(\hat \mu_s - \delta_n, \hat \mu_s + \delta_n), \quad s = 1,\ldots, S, \]
where $\delta_n$ is a sequence of positive constants to be determined.    
\end{itemize}  
To summarize, we have an empirical prior $\Pi_n$ for $\theta = (\omega,\mu)$, supported on the sieve $\Theta_n = \Delta(S) \times \RR^S$, where $S=S_n$ will be specified, with density function 
\[ \pi_n(\theta) = \dir_S(\omega \mid \hat\alpha) \times \prod_{s=1}^S \unif(\mu_s \mid \hat \mu_s - \delta_n, \hat \mu_s + \delta_n). \]
This determines an empirical prior for the density function through the mapping \eqref{eq:finite.mixture}.  

%Our goal here is to show that, for a given rate $\eps_n$, the conditions of Theorem~\ref{thm:rate1} can be satisfied by choosing $c=c_n$ and $J=J_n$ appropriately.  It turns out that Conditions LP1 and GP1 require $J$ to be small, while condition \eqref{eq:lik.bound1}, which boils down to the sieve having good approximation properties in terms of the Kullback--Leibler divergence $K$ and the corresponding second moment $V$ in \eqref{eq:kl.V}, requires $J$ to be large.  So the challenge is to balance the two competing conditions.  

\begin{prop}
\label{prop:mixture1}
Suppose that the true mixing distribution $\theta^\star$ in \eqref{eq:mixture} has compact support.  Set $\eps_n = (\log n)^{1/2} n^{-1/2}$.  If $S_n \propto n\eps_n^2 (\log n)^{-1} = \log n$, $B_n \propto \log^{1/2}(\eps_n^{-1})$, $c_n = n \eps_n^{-2} = n^2/(\log n)^2$, and $\delta_n \propto \eps_n$, then there exists $M > 0$ such that the posterior $\Pi^n$, with $\alpha=1$, corresponding to the empirical prior described above satisfies 
\[ \E_{\theta^\star}^n[ \Pi^n(\{\theta \in \Theta_n: H(p_{\theta^\star}, p_\theta) > M \eps_n\})] \to 0. \]
\end{prop} 

\begin{proof}
See the Appendix.  
\end{proof}

\begin{remark}
\label{re:kernel}
The proof of Proposition~\ref{prop:mixture1} is not especially sensitive to the choice of kernel.  More specifically, the local prior support condition, LP1, can be verified for kernels other than normal, the key condition being Equation~\eqref{jenn} in the Appendix.  For example, that condition can be verified for the Cauchy kernel 
\[ k(x \mid \mu) = \frac{1}{\sigma\pi}\Bigl\{ 1+\frac{(x-\mu)^2}{\sigma^2} \Bigr\}^{-1}, \]
where $\sigma$ is a fixed scale parameter.  Therefore, using the same empirical prior formulation as for the normal case, the same argument in the proof of Proposition~\ref{prop:mixture1} shows that the Cauchy mixture posterior achieves the rate $\eps_n = (\log n) n^{-1/2}$ when the true density $p^\star = p_{\theta^\star}$ is a finite Cauchy mixture.  To our knowledge, this mixture of heavy-tailed kernels has yet to be considered in Bayesian nonparametrics literature, 
%\citep[cf.,][p.~1229]{kruijer.rousseau.vaart.2010}, 
but it fits quite easily into our general setup proposed here.  If bounds on the error in approximating an infinite Cauchy mixture by a finite Cauchy mixture were available, then our analysis immediately gives a rate for the more general case.  
%See, also, \citet{ebmono} for a similar analysis in monotone density estimation.  
\end{remark}

\subsection{Estimation of a sparse normal mean vector}
\label{SS:mean}

Consider inference on the mean vector $\theta = (\theta_1,\ldots,\theta_n)^\top$ of a normal distribution, $\nm_n(\theta, I_n)$, based on a single sample $X=(X_1,\ldots,X_n)$.  That is, $X_i \sim \nm(\theta_i, 1)$, for $i=1,\ldots,n$, independent.  The mean vector is assumed to be sparse in the sense that most of the components, $\theta_i$, are zero, but the locations and values of the non-zero components are unknown.  This problem was considered by \citet{martin.walker.eb} and they show that a version of the double empirical Bayes posterior contracts at the optimal minimax rate.  Here we propose an arguably simpler empirical prior and demonstrate the same asymptotic optimality of the posterior based on the general results in Section~\ref{SS:unknown}.  

Write the mean vector $\theta$ as a pair $(S, \theta_S)$, where $S \subseteq \{1,2,\ldots,n\}$ identifies the non-zero entries of $\theta$, and $\theta_S$ is the $|S|$-vector of non-zero values.  Assume that the true mean vector $\theta^\star$ has $|S_n^\star|=s_n^\star$ such that $s_n^\star = o(n)$.  The sieves $\Theta_{n,S}$ are subsets of $\RR^n$ that constrain the components of the vectors corresponding to indices in $S^c$ to be zero; no constraint on the non-zero components is imposed.  Note that we can trivially restrict to subsets $S$ of cardinality no more than $T_n = n$.  Furthermore, Condition~S2 is trivially satisfied because $\theta^\star$ belongs to the sieve $S_n^\star$ by definition, so we can take $\theta^\dagger = \theta^\star$.  

For this model, the Hellinger distance for joint densities satisfies
\[ H^2(p_{\theta^\star}^n, p_\theta^n) = 1 - e^{-\frac18 \|\theta-\theta^\star\|^2}, \]
where $\|\cdot\|$ is the usual $\ell_2$-norm on $\RR^n$.  In this sparse setting, as demonstrated by \citet{donoho1992}, the $\ell_2$-minimax rate of convergence is $s_n^\star \log(n / s_n^\star)$; we set this rate equal to $n\eps_n^2$, so that $\eps_n^2 = (s_n^\star / n) \log(n / s_n^\star)$.  Therefore, if we can construct a prior such that Conditions~LP2 and GP2 hold for this $\eps_n$, then it will follow from Theorem~\ref{thm:rate2} that the corresponding empirical Bayes posterior concentrates at the optimal minimax rate.  

Let the prior distribution $w_n$ for $S$ be given by 
\[ w_n(S) \propto \binom{n}{|S|}^{-1} e^{-g(|S|) |S|}, \quad S \subseteq \{1,2,\ldots,n\}, \]
where $g(s)$ is a non-decreasing slowly varying function as $s \to \infty$, which includes the case where $g(s) \equiv B$ for a sufficiently large constant $B$; see the proof of the proposition.  For the conditional prior for $\theta_S$, given $S$, we let 
\[ \theta_S \mid S \sim \nm_{|S|}(\hat\theta_{n,S}, \gamma^{-1} I_{|S|}), \quad \text{for any $\gamma \in (0,1)$}, \]
where the sieve MLE is $\hat\theta_{n,S} = X_S = (X_i: i \in S)$.  

\begin{prop}
\label{prop:mean}
Suppose the normal mean vector $\theta^\star$ is $s_n^\star$-sparse in the sense that only $s_n^\star = o(n)$ of the entries in $\theta^\star$ are non-zero.  For the empirical prior described above, there exists a constant $M > 0$ such that the corresponding posterior distribution $\Pi^n$, using Type~I or Type~II regularization, with any $\alpha < 1$, satisfies 
\[ \E_{\theta^\star}^n \bigl[ \Pi^n(\{\theta: \|\theta-\theta^\star\|^2 > M s_n^\star \log(n / s_n^\star)\}) \bigr] \to 0. \]
\end{prop}

\begin{proof}
See the Appendix.
\end{proof}

Note that the prior being employed in this empirical Bayes formulation is conjugate, leading to some computational savings compared to the non-conjugate priors shown to be optimal in \citet{castillo.vaart.2012} under a classical Bayesian formulation; see \citet{martin.mess.walker.eb} and \citet{martin.horseshoe.discuss} for more on computational benefits, and \citet{ebcvg} for results on coverage of credible sets based on this empirical Bayes model.  A similar approach to the one described above is considered in \citet{ebpiece} to get minimax optimal posterior concentration rates and fast computation for the case where $\theta$ is known to be piecewise constant.

\subsection{Regression function estimation}
\label{SS:np.reg}

Consider a nonparametric regression model 
\[ Y_i = f(t_i) + \sigma z_i, \quad i=1,\ldots,n, \]
where $z_1,\ldots,z_n$ are iid~$\nm(0,1)$, $t_1,\ldots,t_n$ are equi-spaced design points in $[0,1]$, i.e., $t_i = i/n$, and $f$ is an unknown function.  Following \citet{arbel.etal.sjs2013}, we consider a Fourier basis expansion for $f=f_\theta$, so that $f(t) = \sum_{j=1}^\infty \theta_j \phi_j(t)$, where $\theta=(\theta_1,\theta_2,\ldots)$ and $(\phi_1,\phi_2,\ldots)$ are the basis coefficients and functions, respectively.  They give conditions such that their Bayesian posterior distribution for $f$, induced by a prior on the basis coefficients $\theta$, concentrates at the true $f^\star$ at the minimax rate corresponding to the unknown smoothness of $f^\star$.  Here we derive a similar result, with a better rate, for the posterior derived from an empirical prior.  

Following the calculations in Section~\ref{SS:mean}, the Hellinger distance between the joint distribution of $(Y_1,\ldots,Y_n)$ for two different regression functions, $f$ and $g$, satisfies 
\[ H^2(p_f^n, p_g^n) = 1 - e^{-\frac{n}{8\sigma^2} \|f-g\|_n^2}, \]
where $\|f\|_n^2 = n^{-1} \sum_{i=1}^n f(t_i)^2$ is the squared $L_2$-norm corresponding to the empirical distribution of the covariate $t$.  So, if the conditions of Theorem~\ref{thm:rate2} are satisfied, then we get a posterior concentration rate result relative to the metric $\|\cdot \|_n$.  

Suppose that the true regression function $f^\star$ is in a Sobolev space of index $\beta > \frac12$.  That is, there is an infinite coefficient vector $\theta^\star$ such that $f^\star = f_{\theta^\star}$ and $\sum_{j=1}^\infty \theta_j^{\star 2} j^{2\beta} \lesssim 1$.  This implies that the coefficients $\theta_j^\star$ for large $j$ are of relatively small magnitude and suggests a particular formulation of the model and empirical prior.  As before, we rewrite the infinite vector $\theta$ as $(S,\theta_S)$, but this time $S$ is just an integer in $\{1,2,\ldots,n\}$, and $\theta_S = (\theta_1,\ldots,\theta_S,0,0,\ldots)$ is an infinite vector with only the first $S$ terms non-zero.  That is, we will restrict our prior to be supported on vectors whose tails vanish in this sense.  For the prior $w_n$ for the integer $S$, we take 
\[ w_n(s) \propto e^{-g(s) s}, \quad s = 1,\ldots,n, \]
where $g(s)$, is a non-decreasing slowly varying function, which includes the case of $g(s) \equiv B$ for $B$ sufficiently large; see the proof of the proposition.  Next, for the conditional prior for $\theta_S$, given $S$, note first that the sieve MLE is a least-squares estimator
\[ \hat\theta_{n,S} = (\Phi_S^\top \Phi_S)^{-1} \Phi_S^\top Y, \]
where $\Phi_S$ is the $n \times |S|$ matrix determined by the basis functions at the observed covariates, i.e., $\Phi_S = (\phi_j(t_i))_{ij}$, $i=1,\ldots,n$ and $j=1,\ldots,|S|$.  As in \citet{martin.mess.walker.eb}, this suggests a conditional prior of the form
\[ \theta_S \mid S \sim \nm_{|S|}\bigl(\hat\theta_{n,S}, \gamma^{-1} (\Phi_S^\top \Phi_S)^{-1} \bigr), \quad \text{for any $\gamma \in (0,1)$}. \]
This empirical prior for $\theta \equiv (S,\theta_S)$ induces a corresponding empirical prior for $f$ through the mapping $\theta \mapsto f_\theta$.

\begin{prop}
\label{prop:np.reg}
Suppose that the true regression function $f^\star$ is in a Sobolev space of index $\beta > \frac12$.  For the empirical prior described above, there exists a constant $M > 0$ such that the corresponding posterior distribution $\Pi^n$, using Type~I or Type~II regularization, with any $\alpha < 1$, satisfies 
\[ \E_{f^\star}^n \bigl[ \Pi^n(\{\theta: \|f_\theta-f^\star\|_n > M n^{-\beta/(2\beta + 1)} \}) \bigr] \to 0. \]
\end{prop}

\begin{proof}
See the Appendix.
\end{proof}

Note that the rate obtained in Proposition~\ref{prop:np.reg} is {\em exactly} the optimal minimax rate, i.e., there are no extra logarithmic factors.  This, like in Section~\ref{SS:mean}, is a consequence of $f^\star$ eventually being in the specified sieve; these extra log factors are a result of having to approximate the true parameter by an element in the sieve.  A similar result, without the additional logarithmic terms, is given in \citet{gao.zhou.2016}.

\subsection{Nonparametric density estimation}
\label{SS:mixture2}

Consider the problem of estimating a density $p$ supported on the real line.  Like in Section~\ref{SS:mixture1}, we propose a normal mixture model and demonstrate the asymptotic concentration properties of the posterior based on an empirical prior, but with the added feature that the rate is adaptive to the unknown smoothness of the true density function.  Specifically, as in \citet{kruijer.rousseau.vaart.2010}, we assume that data $X_1,\ldots,X_n$ are iid from a true density $p^\star$, where $p^\star$ satisfies the conditions C1--C4 in their paper; in particular, we assume that $\log p^\star$ is H\"older with smoothness parameter $\beta$.  They propose a fully Bayesian model---one that does not depend on the unknown $\beta$---and demonstrate that the posterior concentration rate, relative to the Hellinger distance, is $\eps_n = (\log n)^t n^{-\beta/(2\beta + 1)}$ for suitable constant $t > 0$, which is within a logarithmic factor of the optimal rate.  

Here we extend the approach presented in Section~\ref{SS:mixture1} to achieve adaptation by incorporating a prior for the number of mixture components, $S$, as well as the $S$-specific kernel variance $\sigma_S^2$ as opposed to fixing their values.  For the prior $w_n$ for $S$, we let 
\[ w_n(S) \propto e^{-D (\log S)^r S}, \quad S=1,\ldots,n, \]
where $r > 1$ and $D > 0$ are specified constants.  Given $S$, we consider a mixture model with $S$ components of the form
\[ p_{S,\theta_S}(\cdot) = \sum_{s=1}^S \omega_{s,S} \, \nm(\cdot \mid \mu_{s,S}, \lambda_S^{-1}), \]
where $\theta_S = (\omega_S, \mu_S, \lambda_S)$, $\omega_S=(\omega_{1,S},\ldots,\omega_{S,S})$ is a probability vector in $\Delta(S)$, $\mu_S = (\mu_{1,S},\ldots,\mu_{S,S})$ is a $S$-vector of mixture locations, and $\lambda_S$ is a precision (inverse variance) that is the same in all the kernels for a given $S$.  We can fit this model to data using, say, the EM algorithm, and produce a given $S$ sieve MLE: $\hat{\omega}_S = (\hat \omega_{1,S},\ldots,\hat \omega_{S,S})$, $\hat \mu_S = (\hat \mu_1,\ldots, \hat \mu_S)$, and $\hat\lambda_S$.  Following our approach in Section~\ref{SS:mixture1}, consider an empirical prior for $\omega_S$ obtained by taking 
\[ \omega_S \mid S \sim \dir_S(\hat\alpha_S) \]
where $\hat\alpha_{s,S} = 1 + c \hat \omega_{s,S}$ and $c=c_S$ is to be determined.  The prior for $\mu_S$ follows the same approach as in Section~\ref{SS:mixture1}, i.e., 
\[ \mu_{S,s} \sim \unif(\hat \mu_{S,s} - \delta, \hat \mu_{S,s} + \delta), \quad s=1,\ldots,S, \quad \text{independent}, \]
where $\delta=\delta_S$ is to be determined.  The prior for $\lambda_S$ is also uniform,
\[ \lambda_S \sim \unif(\hat\lambda_S(1-\psi),\hat\lambda_S(1+\psi)), \]
where $\psi=\psi_S$ is to be determined.  Also, as with $\hat{\mu}_S$ being restricted to the interval $(-B,+B)$, we restrict the $\hat\lambda_S$ to lie in $(B_l,B_u)$, to be determined.  %but chosen such that for all large $n$ we have the components of $p_n^\dagger\in m(S_n^*)$, where $S_n^*$ is given in the proof of proposition~\ref{prop:mixture2}
Then we get a prior on the density function through the mapping $(S,\theta_S) \mapsto p_{S,\theta_S}$.  For this choice of empirical prior, the following proposition shows that the corresponding posterior distribution concentrates around a suitable true density $p^\star$ at the optimal rate, up to a logarithmic factor, exactly as in \citet{kruijer.rousseau.vaart.2010}.  

\begin{prop}
\label{prop:mixture2}
Suppose that the true density $p^\star$ satisfies Conditions C1--C4 in \citet{kruijer.rousseau.vaart.2010}, in particular, $\log p^\star$ is H\"older continuous with smoothness parameter $\beta$.  For the empirical prior described above, if $B = (\log n)^2$, $B_l = n^{-1}$, $B_u = n^{b-2}$, and, for each $S$, $c=c_s = n^2 S^{-1}$, $\delta=\delta_S = S^{1/2} n^{-(b + 3/2)}$, and $\psi = \psi_S = S n^{-1}$, for a sufficiently large $b > 2$, then there exists constants $M > 0$ and $t > 0$ such that the corresponding posterior distribution $\Pi^n$, using Type~I or Type~II regularization, with any $\alpha < 1$, satisfies 
\[ \E_{p^\star}^n \bigl[ \Pi^n(\{\theta: H(p^\star, p_\theta) > M (\log n)^t n^{-\beta/(2\beta + 1)} \}) \bigr] \to 0. \]
\end{prop}

\begin{proof}
See the Appendix.
\end{proof}

\section{Conclusion}
\label{S:discuss}

This paper considers the construction of an empirical or data-dependent prior such that, when combined with the likelihood via Bayes's formula, gives a posterior distribution with desirable asymptotic concentration properties.  The details vary a bit depending on whether the complexity of the true $\theta^\star$ is known to the user or not (Sections~\ref{SS:known}--\ref{SS:unknown}), but the basic idea is to first choose a suitable sieve and then center the prior for the sieve parameters on the sieve MLE.  This makes establishing the necessary local prior support condition and lower-bounding the posterior denominator straightforward, which is a major obstacle in the standard Bayesian nonparametric setting.  Having the data involved in the prior complicates the usual argument to upper-bound the posterior numerator, but compared to the usual global prior conditions involving entropy, here we only need to suitably control the spread of the empirical prior.  The end result is a data-dependent measure that achieves a certain---often optimal---concentration rate, adaptively, if necessary.  

The approach presented here is quite versatile, so there are many potential applications beyond those examples studied here.  A more general question to be considered in a follow-up work, one that has attracted a lot of attention in the Bayesian nonparametric community recently, concerns the coverage probability of credible regions derived from our empirical Bayes posterior distribution.  Having suitable concentration rates is an important first step, but coverage properties will require new insights.  The theoretical results presented in \citet{ebcvg} for the sparse normal means problem and the numerical results in \citet{ebmono} for the monotone density estimation are both promising, but more work is needed.

\section*{Acknowledgments}

The authors are grateful to the associate editor and anonymous referees for their detailed comments and suggestions. This work is partially supported by the U.~S.~National Science Foundation, grants DMS--1506879 and DMS--1737933.

\appendix

\section{Details for the examples}
%\label{S:proofs}

\subsection{Proof of Proposition~\ref{prop:finite.dim}}

For Condition~LP1, under the proposed normal prior, we have
\[ \Pi_n(\L_n) = \int_{n\,(\theta - \hat\theta_n)^\top \Psi (\theta - \hat\theta_n) < a} \nm\bigl(\theta \mid \hat\theta_n, n^{-1}\Psi^{-1} \bigr) \,d\theta. \]
Making a change of variable, $z = n^{1/2}\Psi^{1/2} (\theta - \hat\theta_n)$, the integral above can be rewritten as 
\[ \Pi_n(\L_n) = \int_{\|z\|^2 < a} \frac{1}{(2\pi)^{d/2}}  e^{-\frac{1}{2} \|z\|^2} \,dz, \]
and, therefore, $\Pi_n(\L_n)$ is lower-bounded by 
a constant not depending on $n$ so $\Pi_n(\L_n)$ is bounded away from zero; hence Condition~LP1 holds with $\eps_n = n^{-1/2}$.  For Condition~GP1, write the prior as $\theta\sim \nm_d(\hat\theta_n,n^{-1}\Psi^{-1})$ and the asymptotic distribution of the MLE as $\hat\theta \sim \nm_d(\theta^\star,n^{-1}\Sigma^{\star-1})$, where $\Sigma^\star$ is the Fisher information matrix evaluated at $\theta^\star$. Then we have,  
$$\pi_n(\theta)^p \propto |pn\Psi|^{-1/2}|n\Psi|^{p/2} \nm_d(\theta \mid \hat\theta_n, (pn\Psi)^{-1}).$$
Thus
$$\E_{\theta^\star}\{\pi_n(\theta)^p\} \propto |pn\Psi|^{-1/2}|n\Psi|^{p/2}\,\nm_d\bigl(\theta \mid \theta^\star,(pn\Psi)^{-1}+n^{-1}\Sigma^{\star-1}\bigr)$$
and so
$$\int \bigl[ \E_{\theta^\star}\{\pi_n(\theta)^p\}\bigr]^{\frac1p}\,d\theta \propto
|I_d + p \Psi \Sigma^{\star-1}|^{\frac12 - \frac{1}{2p}}. 
%|pn\Psi|^{-\frac{1}{2p}}|n\Psi|^{\frac12}\,|(pn\Psi)^{-1}+n^{-1}\Sigma^{*\,-1}|^{-\frac{1}{2p}}\,|p[(pn\Psi)^{-1}+n^{-1}\Sigma^{*\,-1}]|^{\frac12}.
$$
As long as $\Psi$ is non-singular, the right-hand side above is not dependent on $n$ and is finite,  
%Plugging in $\Psi = \gamma \Sigma_n$, we can simplify the above expression (with proportionality constants) as 
%\begin{equation}
%\label{eq:exp.dim}
%\int \bigl[ \E_{\theta^\star}\{\pi_n(\theta)^p\}\bigr]^{\frac1p}\,d\theta = \Bigl( \frac{p\gamma + 1}{\{(2\pi)^{p-1}(p\gamma + p)\}^{1/p}} \Bigr)^{d/2}. 
%\end{equation}
%The key observation is that the integral does not depend on $\Sigma_n$ and, therefore, is bounded in $n$, 
which implies we can take $\eps_n = n^{-1/2}$.  It follows from Theorem~\ref{thm:rate1} that the Hellinger rate is $\eps_n=n^{-1/2}$ and, since all metrics on the finite-dimensional $\Theta$ are equivalent, the same rate obtains for any other metric.  

We should highlight the result 
%in \eqref{eq:exp.dim}, i.e., 
that the integral involved in checking Condition~GP1 is at most exponential in the dimension of the parameter space:
\begin{equation}\label{eq:exp:dim}
 \int \bigl[ \E_{\theta^\star}\{\pi_n(\theta)^p\}\bigr]^{\frac1p}\,d\theta \leq e^{\kappa d}, \quad \kappa > 0. 
 \end{equation}
This result will be useful in the proof of some of the other propositions.

\subsection{Proof of Proposition~\ref{prop:histogram}}

We start by verifying Condition~LP1. Note that, for those models in the support of the prior, the data are multinomial, so the likelihood function is 
\[ L_n(\theta) = \theta_1^{n_1} \cdots \theta_S^{n_S}, \]
where $(n_1,\ldots,n_S)$ are the bin counts, i.e., $n_s = |\{i: X_i \in E_s\}|$, $s=1,\ldots,S$.  Taking expectation with respect to $\theta \sim \dir_S(\hat\alpha)$ gives 
\begin{align*}
\E(\theta_1^{n_1} \cdots \theta_S^{n_S}) & = \frac{\Gamma(c+S)}{\Gamma(c+S+n)} \prod_{s=1}^S \frac{\Gamma(n_s + 1 + c \hat\theta_s)}{\Gamma(1 + c\hat\theta_s)} \\ 
& = \frac{\Gamma(c+S)}{\Gamma(c+S+n)} \prod_{s=1}^S \prod_{k=1}^{n_s} (k + c \hat\theta_s) \\
& \geq \frac{\Gamma(c+S)}{\Gamma(c+S+n)} \prod_{s=1}^S (1+c\hat \theta_s)^{n_s} \\
& \geq \frac{\Gamma(c+S) \, c^n}{\Gamma(c+S+n)} \prod_{s=1}^S \hat \theta_s^{n_s}.
\end{align*}
Therefore, 
\begin{equation}
\label{eq:hist.lr.bound}
\E\{L_n(\theta)\} \geq \frac{\Gamma(c+S) \, c^n}{\Gamma(c+S+n)} L_n(\hat\theta). 
\end{equation}
Next, a simple ``reverse Markov inequality''  says, for any random variable $Y \in (0,1)$,  
\begin{equation}
\label{eq:reverse}
\prob(Y > a) \geq \frac{\E(Y) - a}{1-a}, \quad a \in (0,1). 
\end{equation}
Recall that $\L_n = \{\theta \in \Theta_n: L_n(\theta) > e^{-d n\eps_n^2} L_n(\hat\theta)\}$ as in \eqref{eq:Ln}, so we can apply \eqref{eq:reverse} to get 
\[ \Pi_n(\L_n) \geq\frac{\E\{L_n(\theta)\}/L_n(\hat\theta)-e^{-d n\eps_n^2}}{1-e^{-d n\eps_n^2}}. \]
Then it follows from \eqref{eq:hist.lr.bound} that 
\[ \Pi_n(\L_n) \geq \frac{\Gamma(c+S) \, c^n}{\Gamma(c+S+n)} - e^{-d n\eps_n^2} \]
and, therefore, Condition~LP1 is satisfied, with $C > d$, if 
\begin{equation}
\label{eq:hist.cond1}
\frac{\Gamma(c + S + n)}{\Gamma(c + S) c^n} \leq e^{d n \eps_n^2}.  
\end{equation}
Towards this, we have
\[ \frac{\Gamma(c + S + n)}{\Gamma(c + S) \, c^n} = \prod_{j=1}^n \Bigl(1 + \frac{S + j }{c} \Bigr)  \leq \Bigl(1 + \frac{S+n+1}{c} \Bigr)^n. \]
So, if $c=n\eps_n^{-2}$ as in the proposition statement, then the right-hand side above is upper-bounded by $e^{n\eps_n^2 (1 + S/n)}$.  Since $S \leq n$, \eqref{eq:hist.cond1} holds for, say, $d > 2$, hence, Condition~LP1.  

Towards Condition~GP1, note that the Dirichlet component for $\theta$ satisfies 
\[ \dir_S(\theta \mid \hat\alpha) \leq \dir_S(\hat \theta \mid \hat\alpha) \approx (c + S)^{c+S+1/2} \prod_{s: n_s>0} \frac{1}{(1 + c \hat \theta_s)^{c \hat \theta_s + 1/2}} \, \hat \theta_s^{c \hat \theta_s}, \]
where the ``$\approx$'' is by Stirling's formula, valid for all $n_s>0$ due to the value of $c$.  This has a uniform upper bound:
\[ \dir_S(\theta \mid \hat\alpha) \leq \frac{(c + S)^{c + S + 1/2}}{c^c}, \quad \forall \; \theta \in \Delta(S). \]
Then Condition~GP1 holds if we can bound the product of this and $\Gamma(S)^{-1}$, the volume of $\Delta(S)$, by $e^{K n \eps_n^2}$ for a constant $K > 0$.  Using Stirling's formula again, and the fact that $c/S \to \infty$, we have 
\[ \frac{(c+S)^{c+S+1/2}}{c^{c+S/2} \,\Gamma(S)} = \frac{S^{S+1/2}}{c^{S/2} \, \Gamma(S)} \Bigl(1 + \frac{S}{c} \Bigr)^c \Bigl(1 + \frac{c}{S} \Bigr)^{S+1/2} \leq e^{K' S \log(1+c/S)}, \quad K' > 0. \]
We need $S\log(1+c/S)\leq n\eps_n^2$.  Since $c/S \ll n^2$, the logarithmic term is $\lesssim \log n$.  But we assumed that $S \leq n\eps_n^2(\log n)^{-1}$, so the product is $\lesssim n \eps_n^2$, proving Condition~GP1.  

It remains to check Condition~S1.  A natural candidate for the pseudo-true parameter $\theta^\dagger$ in Condition~S1 is one that sets $\theta_s$ equal to the probability assigned by the true density $p^\star$ to $E_s$.  Indeed, set 
\[ \theta_s^\dagger = \int_{E_s} p^\star(x) \,dx, \quad s=1,\ldots,S. \]
It is known \citep[e.g.,][p.~93]{scricciolo2015} that, if $p^\star$ is $\beta$-H\"older, with $\beta \in (0,1]$, then the sup-norm approximation error of $p_{\theta^\dagger}$ is 
\[ \|p^\star - p_{\theta^\dagger}\|_\infty \lesssim S^{-\beta}. \]
Since $p^\star$ is uniformly bounded away from 0, it follows from Lemma~8 in \citet{ghosalvaart2007} that $\max\{K(p^\star, p_{\theta^\dagger}),V(p^\star, p_{\theta^\dagger})\} \lesssim H^2(p^\star, p_{\theta^\dagger})$ which, in turn, is upper-bounded by $S^{-2\beta}$ by the above display.  Therefore, we need $S=S_n$ to satisfy $S^{-\beta} \leq \eps_n$, and this is achieved by choosing $S=n\eps_n^2 (\log n)^{-1}$ as in the proposition.  This establishes Condition~S1, completing the proof.

\subsection{Proof of Proposition~\ref{prop:mixture1}}

We start by verifying Condition~LP1.  Towards this, we first note that, for mixtures in the support of the prior, the likelihood function is 
\[ L_n(\theta)=\prod_{i=1}^n \sum_{s=1}^S \omega_s \, k(X_i \mid \mu_s), \quad \theta = (\omega, \mu), \]
which can be rewritten as
\begin{equation}
\label{eq:mix.lik}
L_n(\theta)=\sum_{(n_1,\ldots,n_S)} \omega_1^{n_1} \cdots \omega_S^{n_S} \sum_{(s_1,\ldots,s_n)}\prod_{s=1}^S \prod_{i: s_i = s} k(X_i \mid \mu_s), 
\end{equation}
where the first sum is over all $S$-tuples of non-negative integers $(n_1,\ldots,n_S)$ that sum to $n$, the second sum is over all $n$-tuples of integers $1,\ldots,S$ with $(n_1,\ldots,n_S)$ as the corresponding frequency table, and $k(x \mid \mu) = \nm(x \mid \mu, \sigma^2)$ for known $\sigma^2$.  We also take the convention that, if $n_s = 0$, then the product $\prod_{i: s_i = s}$ is identically 1.  Next, since the prior has $\omega$ and $\mu$ independent, we only need to bound  
\[ \E(\omega_1^{n_1} \cdots \omega_S^{n_S}) \quad \text{and} \quad \E\Bigl\{ \prod_{s=1}^S \prod_{i: s_i = s} k(X_i \mid \mu_s) \Bigr\} \]
for a generic $(n_1,\ldots,n_S)$.  The first expectation is with respect to the prior for $\omega$ and can be handled exactly like in the proof of Proposition~\ref{prop:histogram}.  For the second expectation, which is with respect to the prior for the $\mu$, since the prior has the components of $\mu$ independent, we have 
\[ \E\Bigl\{ \prod_{s=1}^S \prod_{i:s_i=s} k(X_i \mid \mu_s) \Bigr\} = \prod_{s=1}^S \E\Bigl\{ \prod_{i: s_i=s} k(X_i \mid \mu_s) \Bigr\}, \]
so we can work with a generic $s$.  Writing out the product of kernels, we get 
\[ \E\Bigl\{ \prod_{i: s_i=s} k(X_i \mid \mu_s) \Bigr\} = \Bigl( \frac{1}{2\pi \sigma^2} \Bigr)^{n_s/2} e^{-\frac{1}{2\sigma^2} \sum_{i: s_i=s} (X_i - \Xbar)^2} \E\bigl\{e^{-\frac{n_s}{2\sigma^2}(\mu_s - \Xbar)^2} \}. \]
By Jensen's inequality, i.e., $\E(e^Z) \geq e^{\E(Z)}$, the expectation on the right-hand side is lower bounded by
\[ e^{-\frac{n_s}{2\sigma^2}\E(\mu_s - \Xbar)^2} = e^{-\frac{n_s}{2\sigma^2}\{v_n + (\hat \mu_s - \Xbar)^2\}}, \]
where $v_n = \delta_n^2/3$ is the variance of $\mu_s \sim \unif(\hat \mu_s - \delta_n, \hat \mu_s + \delta_n)$. This implies 
\begin{equation}\label{jenn}
\E\Bigl\{ \prod_{s=1}^S \prod_{i: s_i=s} k(X_i \mid \mu_s) \Bigr\} \geq e^{-\frac{n v_n}{2\sigma^2}} \prod_{s=1}^S \prod_{i:s_i=s} k(X_i \mid \hat \mu_s). 
\end{equation}
Putting the two expectations back together, from \eqref{eq:mix.lik} we have that 
\begin{equation}
\label{eq:mixture.lr.bound}
\E\{L_n(\theta)\} \geq \frac{\Gamma(c+S) \, c^n}{\Gamma(c+S+n)} e^{-\frac{n v_n}{2\sigma^2}} L_n(\hat\theta) 
\end{equation}
where now the expectation is with respect to both priors.
Recall that $\L_n = \{\theta \in \Theta_n: L_n(\theta) > e^{-d n\eps_n^2} L_n(\hat\theta)\}$ as in \eqref{eq:Ln}, and define $\L_n' = \{\theta \in \L_n: L_n(\theta) \leq L_n(\hat\theta_n)\}$.  Since, $\L_n \supseteq \L_n'$ and, for $\theta \in \L_n'$, we have $L_n(\theta)/L_n(\hat\theta_n) \leq 1$, we can apply the reverse Markov inequality \eqref{eq:reverse} again to get 
\[ \Pi_n(\L_n) \geq\frac{\E\{L_n(\theta)\}/L_n(\hat\theta)-e^{-d n\eps_n^2}}{1-e^{-d n\eps_n^2}}. \]
Then it follows from \eqref{eq:mixture.lr.bound} that 
\[ \Pi_n(\L_n) \geq \frac{\Gamma(c+S) \, c^n}{\Gamma(c+S+n)}e^{-\frac{n v_n}{2\sigma^2}} - e^{-d n\eps_n^2} \]
and, therefore, Condition~LP1 is satisfied if 
\begin{equation}
\label{eq:mixture.cond1}
\frac{n v_n}{2\sigma^2} \leq b n \eps_n^2 \quad \text{and} \quad \frac{\Gamma(c + S + n)}{\Gamma(c + S) c^n} \leq e^{a n \eps_n^2},  
\end{equation}
where $a + b < d$.  The first condition is easy to arrange; it requires that  
\[ v_n \leq 2b \sigma^2 \eps_n^2 \iff \delta_n \leq (6 b \sigma^2)^{1/2} \eps_n, \]
which holds by assumption on $\delta_n$.  The second condition holds with $a=2$ by the argument presented in the proof of Proposition~\ref{prop:histogram}.  Therefore, Condition~LP1 holds.  

Towards Condition~GP1, putting together the bound on the Dirichlet density function in the proof of Proposition~\ref{prop:histogram} and the following bound on the uniform densities, 
\[ \prod_{s=1}^S \unif(\mu_s \mid \hat \mu_s - \delta_n, \hat \mu_s + \delta_n) \leq \Bigl( \frac{1}{2\delta_n} \Bigr)^S \prod_{s=1}^S I_{[-B_n - \delta_n, B_n + \delta_n]}(\mu_s), \]
we have that, for any $p > 1$, 
\[ \int_{\Theta_n} \bigl[ \E_{\theta^\star}\{\pi_n(\theta)^p \} \bigr]^{1/p} \,d\theta \leq \frac{(c + S)^{c + S + 1/2}}{c^c \, \Gamma(S)} \cdot \Bigl( \frac{1}{2\delta_n} \Bigr)^S \{ 2(B_n + \delta_n) \}^S. \]
Then Condition~GP1 holds if we can make both terms in this product to be like $e^{K n \eps_n^2}$ for a constant $K > 0$.  The first term in the product, coming from the Dirichlet part, is handled just like in the proof of Proposition~\ref{prop:histogram} and, for the second factor, we have 
\[ \Bigl( \frac{1}{2\delta_n} \Bigr)^S \{ 2(B_n + \delta_n) \}^S \leq e^{S \log(1 + \frac{B_n}{\delta_n})}. \]
Since $\delta_n \propto \eps_n$ and $B_n \propto \log^{1/2}(\eps_n^{-1})$, we have  $B_n / \delta_n \propto n^{1/2}$, so the exponent above is $\lesssim S \log n \lesssim n\eps_n^2$.  This takes care of the second factor, proving Condition~GP1.  

Finally, we refer to Section~4 in \citet{ghosalvaart2001} where they show that there exists a finite mixture, characterized by $\theta^\dagger$, with $S$ components and locations in $[-B_n,B_n]$, such that $\max\{ K(p_{\theta^\star}, p_{\theta^\dagger}), \, V(p_{\theta^\star}, p_{\theta^\dagger})\} \leq \eps^2$.  This $\theta^\dagger$ satisfies our Condition~S1, so the proposition follows from Theorem~\ref{thm:rate1}.   

In the context of Remark~\ref{re:kernel}, when the normal kernel is replaced by a Cauchy kernel, we need to verify \eqref{jenn} in order to meet LP1.  To this end, let us start with 
$$\E \exp\left[-\log \prod_{s_i=s} \bigl\{ 1+(X_i-\mu_s)^2/\sigma^2 \bigr\} \right]$$
where the expectation is with respect to the prior for the $\mu_s$ and the $\sigma$ is assumed known.
This expectation is easily seen to be lower-bounded by
$$\exp\left\{-\sum_{s_i=s}\log [1+\E(X_i-\mu_s)^2/\sigma^2]\right\}=
\exp\left\{-\sum_{s_i=s}\log [1+(X_i-\hat{\mu}_s)^2/\sigma^2+v_n/\sigma^2]\right\}.$$
The right-hand term term can be written as
$$\left\{\prod_{s_i=s}\frac{1}{1+(X_i-\hat{\mu}_s)^2/\sigma^2}\right\}\,\frac{1}{\prod_{s_i=s} \left(1+\frac{v_n/\sigma^2}{1+(X_i-\hat{\mu}_s)^2/\sigma^2}\right)}$$
and the second term here is lower-bounded by $\exp(-n_s\,v_n/\sigma^2)$.  Therefore, Condition LP1 holds with the same $\eps_n$ as in the normal case.  

Condition GP1 in this case does not depend on the form of the kernel, whether it be normal or Cauchy.  And S1 is satisfied if we assume the true density $p^\star=p_{\theta^\star}$ is a finite mixture of densities, for example, the Cauchy.  This proves the claim in Remark~\ref{re:kernel}, namely, that the empirical Bayes posterior, based on a Cauchy kernel, concentrates at the rate $\eps_n = (\log n) n^{-1/2}$ when the true density is a finite Cauchy mixture.

\subsection{Proof of Proposition~\ref{prop:mean}}

The proportionality constant depends on $n$ (and $g$) but it is bounded away from zero and infinity as $n \to \infty$ so can be ignored in our analysis.  Here we can check the second part of Condition~LP2.  Indeed, for the true model $S_n^\star$ of size $s_n^\star$, using the inequality $\binom{n}{s} \leq (e n / s)^s$, we have 
\[ w_n(S_n^\star) \propto \binom{n}{s_n^\star}^{-1} e^{-B s_n^\star} \geq e^{-[B + 1 + \log(n/s^\star)] s_n^\star} \]
and, since $n\eps_n^2 = s_n^\star \log(n / s_n^\star)$, the second condition in Condition~LP2 holds for all large $n$ with $A > 1$.  Next, for Condition~GP2, note that the prior $w_n$ given above corresponds to a hierarchical prior for $S$ that starts with a truncated geometric prior for $|S|$ and then a uniform prior for $S$, given $|S|$.  Then it follows directly that Condition~GP2 on the marginal prior for $|S|$ is satisfied.

For Condition~LP2, we first write the likelihood ratio for a generic $\theta \in \Theta_S$:
\[ \frac{L_n(\theta)}{L_n(\hat\theta_{n,S})} = e^{-\frac12 \|\theta_S - \hat\theta_{n,s}\|^2}. \]
Therefore, $\L_{n,S} = \{\theta \in \Theta_S: \frac12 \|\theta - \hat\theta_{n,S}\|^2 < |S|\}$.  This is just a ball in $\RR^{|S|}$ so we can bound the Gaussian measure assigned to it.  Indeed, 
\begin{align*}
\Pi_n(\L_{n,S}) & = \int_{\|z\|^2 < 2|S|} (2\pi)^{-d/2} \gamma^{d/2} e^{-\frac{\gamma}{2} \|z\|^2} \,dz \\
& > (2\pi)^{-|S|/2} \gamma^{|S|/2} e^{-\gamma|S|} \frac{\pi^{|S|/2}}{\Gamma(\frac{|S|}{2} + 1)} (2|S|)^{|S|/2} \\
& = \gamma^{|S|/2} e^{-\gamma |S|} \frac{1}{\Gamma(\frac{|S|}{2} + 1)} |S|^{|S|/2}.
\end{align*} 
Stirling's formula gives an approximation of the lower bound:
\[ e^{-\gamma |S|} \gamma^{|S|/2} 2^{|S|/2} e^{|S|/2}\Bigl( \frac{|S|/2}{2\pi} \Bigr)^{1/2}. \]
For moderate to large $|S|$, the above display is $\gtrsim \exp\bigl\{\bigl( 1 - 2\gamma + \log\gamma + \log 2 \bigr) \tfrac{|S|}{2} \bigr\}$ and, therefore, plugging in $S_n^\star$ for the generic $S$ above, we see that Condition~LP2 holds if $1-2\gamma+\log\gamma+\log 2<0$.  For Condition~GP2, the calculation is similar to that in the finite-dimensional case handled in Proposition~\ref{prop:finite.dim}.  Indeed, the last part of the proof showed that, for a $d$-dimensional normal mean model with covariance matrix $\Sigma^{-1}$ and a normal empirical prior of with mean $\hat\theta_n$ and covariance matrix proportional to $\Sigma^{-1}$, then the integral specified in the second part of Condition~GP2 is exponential in the dimension $d$.  In the present case, we have that  
\[ \int_{\Theta_S} \bigl[ \E_{\theta^\star}\{\pi_{n,S}(\theta)^p\}\bigr]^{\frac1p}\,d\theta = e^{\kappa |S|} \]
for some $\kappa>0$ and then, clearly, Condition~GP2 holds with $K =\kappa$. 
If we take $B$ in the prior $w_n$ for $S$ to be larger than this $K$, then the conditions of Theorem~\ref{thm:rate2} are met with $\eps_n^2 = (s_n^\star / n) \log(n / s_n^\star)$.  

%This implies 
%\[ \E_{\theta^\star}[\Pi^n(\{\theta \in \RR^n: \|\theta-\theta^\star\|^2 > M' s_n^\star \log(n / s_n^\star)\})] \to 0, \]
%and, therefore, the posterior has the optimal minimax concentration rate with respect to the $\ell_2$-norm.  Moreover, this rate is considered adaptive because it is achieved even though the prior was not aware of the true model size $s_n^\star$.  

\subsection{Proof of Proposition~\ref{prop:np.reg}}

By the choice of marginal prior for $S$ and the normal form of the conditional prior for $\theta_S$, given $S$, Conditions~LP2 and GP2 follow almost exactly like in the proof of Proposition~\ref{prop:mean}.  Indeed, the second part of Condition~GP2 holds with $K$ the same as was derived above.  Therefore, we have only to check Condition~S2. Let $p_\theta$ denote the density corresponding to regression function $f=f_\theta$.  If $\theta^\star$ is the coefficient vector in the basis expansion of $f^\star$, then it is easy to check that 
\[ K(p_{\theta^\star}^n, p_{\theta_S^\star}^n) = \frac{n}{2\sigma^2} \|\theta^\star - \theta_S^\star\|^2 = \frac{n}{2\sigma^2} \sum_{j > |S|} \theta_j^{\star 2}. \]
If $f^\star$ is smooth in the sense that it belongs to a Sobolev space indexed by $\beta > \tfrac12$, i.e., the basis coefficient vector $\theta^\star$ satisfies $\sum_{j=1}^\infty \theta_j^{\star 2} j^{2\beta} \lesssim 1$, then it follows that 
\[ K(p_{\theta^\star}^n, p_{\theta_S^\star}^n) \lesssim n |S|^{-2\beta}. \]
So, if we take $\eps_n = n^{-\beta/(2\beta + 1)}$ and $|S_n^\star| = \lfloor n \eps_n^2 \rfloor = \lfloor n^{1/(2\beta + 1)} \rfloor$, then a candidate $\theta^\dagger$ in Condition~S2 is $\theta^\dagger = \theta_S^\star$.  That the desired bound on the Kullback--Leibler second moment $V$ also holds for this $\theta^\dagger$ follows similarly, as in \citet[][p.~558]{arbel.etal.sjs2013}.  This establishes Condition~S2 so the conclusion of the proposition follows from Theorem~\ref{thm:rate2}.

\subsection{Proof of Proposition~\ref{prop:mixture2}}

Write $\eps_n = (\log n)^t n^{-\beta / (2\beta + 1)}$ for a constant $t > 0$ to be determined.  For Condition~S2, we appeal to Lemma~4 in \citet{kruijer.rousseau.vaart.2010} which states that there exists a finite normal mixture, $p^\dagger$, having $S_n^\star$ components, with 
\[ S_n^\star \lesssim n^{1/(2\beta + 1)} (\log n)^{k-t} = n\eps_n^2 (\log n)^{k-3t}, \]
such that $\max\bigl\{ K(p^\star, p^\dagger), \, V(p^\star, p^\dagger) \bigr\} \leq \eps_n^2$, where $k=2/\tau_2$ and $\tau_2$ is related to the tails of $p^\star$ in their Condition~C3.  So, if $t$ is sufficiently large, then our Condition~S2 holds.   

For Condition~GP2, we first note that, by a straightforward modification of the argument given in the proof of Proposition~\ref{prop:mixture1}, we have 
\[ \int_{\Delta(S)\times \RR^S\times \RR_+} \bigl[ \E_{p^\star}\{\pi_{n,S}(\theta)^p\} \bigr]^{1/p} \,d\theta \leq e^{b S \log n} \Bigl( 1+ \frac{B}{\delta} \Bigr)^S \frac{B_u(1+\psi)-B_l(1-\psi)}{2\psi B_l},\]
for some $b>0$.
The logarithmic term appears in the first product because, as in the proof of Proposition~\ref{prop:mixture1}, the exponent can be bounded by a constant times $S \log(1 + c/S) \lesssim S \log n$ since $c/S = n^2 / S^2 < n^2$.  To get the upper bound in the above display to be exponential in $S$, we can take 
\[ \delta \gtrsim \frac{B}{n^b} \quad \text{and} \quad \psi \gtrsim \frac{B_u-B_l}{B_l}\frac{1}{e^{bS\log n}-(B_l+B_u)/(2B_l)}. \]
With these choices, it follows that the right-hand side in the previous display is upper bounded by $e^{3b\log n}$, independent of $S$.  Therefore, trivially, the summation in \eqref{eq:gp2.sum} is also upper bounded by $e^{3b\log n}$.  Since $\log n \leq n \eps_n^2$, we have that Condition~GP2 holds.  

Condition~LP2 has two parts to it.  For the first part, which concerns the prior concentration on $\L_n$, we can follow the argument in the proof of Proposition~\ref{prop:mixture1}.  In particular, with the additional prior on $\lambda$, the corresponding version of \eqref{eq:mixture.lr.bound} is
$$\E L_n(\theta_S)\geq \frac{\Gamma(c+S) \, c^n}{\Gamma(c+S+n)} e^{-\frac16 n \delta^2 \hat\lambda} \,e^{-nz\psi}L_n(\hat\theta_S)$$
for some $z \in (0,1)$.  This is based on the result that if $\lambda\sim \unif(\hat\lambda(1-\psi),\hat\lambda(1+\psi))$ then $\E\lambda=\hat\lambda$ and $\E\log\lambda>\log\hat\lambda-z\psi$ for some $z \in (0,1)$.  With $c = n^2 S^{-1}$ as proposed, the argument in the proof of Proposition~\ref{prop:histogram} shows that the first term on the right-hand side of the above display is lower-bounded by $e^{-CS}$ for some $C > 0$.  To make other other terms lower-bounded by something of the order $e^{-C'S}$, we need $\delta$ and $\psi$ to satisfy 
\[ \delta^2 \lesssim \frac{1}{B_u^2} \, \frac{S}{n} \quad \text{and} \quad \psi \lesssim \frac{S}{n}. \]
Given these constraints and those coming from checking Condition~GP2 above, we require
\[ \frac{B}{n^b} \lesssim \frac{1}{B_u} \Bigl( \frac{S}{n} \Bigr)^{1/2} \quad \text{and} \quad n^{bS}-\frac12 \Bigl(1+\frac{B_u}{B_l} \Bigr) \lesssim \frac{n \, B_u}{B_l}. \]
From Lemma~4 in \citet{kruijer.rousseau.vaart.2010}, we can deduce that the absolute value of the locations for $p^\dagger$ are smaller than a constant times $\log \eps_n^{-\beta}$.  Hence, we can take $B = (\log n)^2$. Also, we need $B_l \lesssim \eps_n^\beta$ which is met by taking $B_l = n^{-1}$.  To meet our constraints, we can take $B_u = n^{b-2}$, so we need $b\geq 2$.  These conditions on $(B, B_l, B_u, \delta, \psi)$ are met by the choices stated in the proposition.  For the second part of Condition~LP2, which concerns the concentration of $w_n$ around $S_n^\star$, we have 
\[ w_n(S_n^\star) \geq e^{-D (\log S_n^\star)^r S_n^\star} \gtrsim e^{-D n\eps_n^2 (\log n)^{k + r - 3t}}. \]
So, just like in \citet{kruijer.rousseau.vaart.2010}, as long as $3t > k + r$, we get $w_n(S_n^\star) \geq e^{-D n\eps_n^2}$ as required in Condition~LP2.

\ifthenelse{1=1}{
\bibliographystyle{apalike}
\bibliography{/Users/rgmarti3/Dropbox/Research/mybib}

\begin{thebibliography}{}

\bibitem[Arbel et~al., 2013]{arbel.etal.sjs2013}
Arbel, J., Gayraud, G., and Rousseau, J. (2013).
\newblock Bayesian optimal adaptive estimation using a sieve prior.
\newblock {\em Scand. J. Stat.}, 40(3):549--570.

\bibitem[Arias-Castro and Lounici, 2014]{ariascastro.lounici.2014}
Arias-Castro, E. and Lounici, K. (2014).
\newblock Estimation and variable selection with exponential weights.
\newblock {\em Electron. J. Stat.}, 8(1):328--354.

\bibitem[Armagan et~al., 2013]{armagan.etal.2013}
Armagan, A., Dunson, D.~B., and Lee, J. (2013).
\newblock Generalized double {P}areto shrinkage.
\newblock {\em Statist. Sinica}, 23(1):119--143.

\bibitem[Barron, 1988]{barron1988}
Barron, A. (1988).
\newblock The exponential convergence of posterior probabilities with
  implications for bayes estimators of density functions.
\newblock Technical Report~7, Department of Statistics, University of Illinois,
  Champaign, IL.

\bibitem[Belitser, 2017]{belitser.ddm}
Belitser, E. (2017).
\newblock On coverage and local radial rates of credible sets.
\newblock {\em Ann. Statist.}, 45(3):1124--1151.

\bibitem[Belitser and Ghosal, 2017]{belitser.ghosal.ebuq}
Belitser, E. and Ghosal, S. (2017).
\newblock Empirical {B}ayes oracle uncertainty quantification.
\newblock Unpublished manuscript,
  \url{http://www4.stat.ncsu.edu/~ghoshal/papers/oracle_regression.pdf}.

\bibitem[Belitser and Nurushev, 2017]{belitser.nurushev.uq}
Belitser, E. and Nurushev, N. (2017).
\newblock Needles and straw in a haystack: robust confidence for possibly
  sparse sequences.
\newblock Unpublished manuscript, {\tt arXiv:1511.01803}.

\bibitem[Berger, 1985]{berger1985}
Berger, J.~O. (1985).
\newblock {\em Statistical Decision Theory and {B}ayesian Analysis}.
\newblock Springer-Verlag, New York, second edition.

\bibitem[Bhadra et~al., 2017]{bhadra.hsplus.2017}
Bhadra, A., Datta, J., Polson, N.~G., and Willard, B. (2017).
\newblock The horseshoe+ estimator of ultra-sparse signals.
\newblock {\em Bayesian Anal.}, 12(4):1105--1131.

\bibitem[Bhattacharya et~al., 2015]{dunson.shrinkage}
Bhattacharya, A., Pati, D., Pillai, N.~S., and Dunson, D.~B. (2015).
\newblock Dirichlet-{L}aplace priors for optimal shrinkage.
\newblock {\em J. Amer. Statist. Assoc.}, 110(512):1479--1490.

\bibitem[Carlin and Louis, 1996]{carlin.louis.1996}
Carlin, B.~P. and Louis, T.~A. (1996).
\newblock {\em Bayes and Empirical {B}ayes Methods for Data Analysis},
  volume~69 of {\em Monographs on Statistics and Applied Probability}.
\newblock Chapman \& Hall, London.

\bibitem[Carvalho et~al., 2010]{carvalho.polson.scott.2010}
Carvalho, C.~M., Polson, N.~G., and Scott, J.~G. (2010).
\newblock The horseshoe estimator for sparse signals.
\newblock {\em Biometrika}, 97(2):465--480.

\bibitem[Castillo and van~der Vaart, 2012]{castillo.vaart.2012}
Castillo, I. and van~der Vaart, A. (2012).
\newblock Needles and straw in a haystack: posterior concentration for possibly
  sparse sequences.
\newblock {\em Ann. Statist.}, 40(4):2069--2101.

\bibitem[Donnet et~al., 2018]{rousseau.etal.eb}
Donnet, S., Rivoirard, V., Rousseau, J., and Scricciolo, C. (2018).
\newblock Posterior concentration rates for empirical {B}ayes procedures with
  applications to {D}irichlet process mixtures.
\newblock {\em Bernoulli}, 24(1):231--256.

\bibitem[Donoho et~al., 1992]{donoho1992}
Donoho, D.~L., Johnstone, I.~M., Hoch, J.~C., and Stern, A.~S. (1992).
\newblock Maximum entropy and the nearly black object.
\newblock {\em J. Roy. Statist. Soc. Ser. B}, 54(1):41--81.
\newblock With discussion and a reply by the authors.

\bibitem[Efron, 2010]{efron2010book}
Efron, B. (2010).
\newblock {\em Large-Scale Inference}, volume~1 of {\em Institute of
  Mathematical Statistics Monographs}.
\newblock Cambridge University Press, Cambridge.

\bibitem[Gao and Zhou, 2016]{gao.zhou.2016}
Gao, C. and Zhou, H.~H. (2016).
\newblock Rate exact {B}ayesian adaptation with modified block priors.
\newblock {\em Ann. Statist.}, 44(1):318--345.

\bibitem[Ghosal et~al., 2000]{ggv2000}
Ghosal, S., Ghosh, J.~K., and van~der Vaart, A.~W. (2000).
\newblock Convergence rates of posterior distributions.
\newblock {\em Ann. Statist.}, 28(2):500--531.

\bibitem[Ghosal and van~der Vaart, 2007a]{ghosalvaart2007a}
Ghosal, S. and van~der Vaart, A. (2007a).
\newblock Convergence rates of posterior distributions for non-i.i.d.
  observations.
\newblock {\em Ann. Statist.}, 35(1):192--223.

\bibitem[Ghosal and van~der Vaart, 2017]{ghosal.vaart.book}
Ghosal, S. and van~der Vaart, A. (2017).
\newblock {\em Fundamentals of Nonparametric {B}ayesian Inference}, volume~44
  of {\em Cambridge Series in Statistical and Probabilistic Mathematics}.
\newblock Cambridge University Press, Cambridge.

\bibitem[Ghosal and van~der Vaart, 2001]{ghosalvaart2001}
Ghosal, S. and van~der Vaart, A.~W. (2001).
\newblock Entropies and rates of convergence for maximum likelihood and {B}ayes
  estimation for mixtures of normal densities.
\newblock {\em Ann. Statist.}, 29(5):1233--1263.

\bibitem[Ghosal and van~der Vaart, 2007b]{ghosalvaart2007}
Ghosal, S. and van~der Vaart, A.~W. (2007b).
\newblock Posterior convergence rates of {D}irichlet mixtures at smooth
  densities.
\newblock {\em Ann. Statist.}, 35(2):697--723.

\bibitem[Kruijer et~al., 2010]{kruijer.rousseau.vaart.2010}
Kruijer, W., Rousseau, J., and van~der Vaart, A. (2010).
\newblock Adaptive {B}ayesian density estimation with location-scale mixtures.
\newblock {\em Electron. J. Stat.}, 4:1225--1257.

\bibitem[Lee et~al., 2017]{lee.lee.lin.deb}
Lee, K., Lee, J., and Lin, L. (2017).
\newblock Minimax posterior convergence rates and model selection consistency
  in high-dimensional {DAG} models based on sparse {C}holesky factors.
\newblock {\em Ann. Statist.}, to appear, {\tt arXiv:1811.06198}.

\bibitem[Martin, 2017]{martin.horseshoe.discuss}
Martin, R. (2017).
\newblock Invited comment on the article by van der {P}as, {S}zab\'o, and van
  der {V}aart.
\newblock {\em Bayesian Anal.}, 12(4):1254--1258.

\bibitem[Martin, 2018]{ebmono}
Martin, R. (2018).
\newblock Empirical priors and posterior concentration rates for a monotone
  density.
\newblock {\em Sankhya A}, to appear, {\tt arXiv:1706.08567}.

\bibitem[Martin et~al., 2017]{martin.mess.walker.eb}
Martin, R., Mess, R., and Walker, S.~G. (2017).
\newblock Empirical {B}ayes posterior concentration in sparse high-dimensional
  linear models.
\newblock {\em Bernoulli}, 23(3):1822--1847.

\bibitem[Martin and Ning, 2018]{ebcvg}
Martin, R. and Ning, B. (2018).
\newblock Empirical priors and coverage of posterior credible sets in a sparse
  normal mean model.
\newblock {\tt arXiv:1812.02150}.

\bibitem[Martin and Shen, 2017]{ebpiece}
Martin, R. and Shen, W. (2017).
\newblock Asymptotically optimal empirical {B}ayes inference in a piecewise
  constant sequence model.
\newblock Unpublished manuscript, {\tt arXiv:1712.03848}.

\bibitem[Martin and Tang, 2019]{ebpred}
Martin, R. and Tang, Y. (2019).
\newblock Empirical priors for prediction in sparse high-dimensional linear
  regression.
\newblock {\tt arXiv:1903.00961}.

\bibitem[Martin and Walker, 2014]{martin.walker.eb}
Martin, R. and Walker, S.~G. (2014).
\newblock Asymptotically minimax empirical {B}ayes estimation of a sparse
  normal mean vector.
\newblock {\em Electron. J. Stat.}, 8(2):2188--2206.

\bibitem[Salomond, 2014]{salomond2014}
Salomond, J.-B. (2014).
\newblock Concentration rate and consistency of the posterior distribution for
  selected priors under monotonicity constraints.
\newblock {\em Electron. J. Stat.}, 8(1):1380--1404.

\bibitem[Scricciolo, 2007]{scricciolo2007}
Scricciolo, C. (2007).
\newblock On rates of convergence for {B}ayesian density estimation.
\newblock {\em Scand. J. Statist.}, 34(3):626--642.

\bibitem[Scricciolo, 2015]{scricciolo2015}
Scricciolo, C. (2015).
\newblock Bayesian adaptation.
\newblock {\em J. Statist. Plann. Inference}, 166:87--101.

\bibitem[Shen and Ghosal, 2015]{shen.ghosal.2015}
Shen, W. and Ghosal, S. (2015).
\newblock Adaptive {B}ayesian procedures using random series priors.
\newblock {\em Scand. J. Stat.}, 42(4):1194--1213.

\bibitem[Shen and Wasserman, 2001]{shen.wasserman.2001}
Shen, X. and Wasserman, L. (2001).
\newblock Rates of convergence of posterior distributions.
\newblock {\em Ann. Statist.}, 29(3):687--714.

\bibitem[Szab{\'o} et~al., 2013]{szabo.vaart.zanten.2013}
Szab{\'o}, B.~T., van~der Vaart, A.~W., and van Zanten, J.~H. (2013).
\newblock Empirical {B}ayes scaling of {G}aussian priors in the white noise
  model.
\newblock {\em Electron. J. Stat.}, 7:991--1018.

\bibitem[van~der Pas et~al., 2017a]{pas.szabo.vaart.rate}
van~der Pas, S., Szab\'o, B., and van~der Vaart, A. (2017a).
\newblock Adaptive posterior contraction rates for the horseshoe.
\newblock {\em Electron. J. Stat.}, 11(2):3196--3225.

\bibitem[van~der Pas et~al., 2017b]{pas.szabo.vaart.uq}
van~der Pas, S., Szab\'{o}, B., and van~der Vaart, A. (2017b).
\newblock Uncertainty quantification for the horseshoe (with discussion).
\newblock {\em Bayesian Anal.}, 12(4):1221--1274.
\newblock With a rejoinder by the authors.

\bibitem[van~der Vaart and van Zanten, 2009]{vaart.zanten.2009}
van~der Vaart, A.~W. and van Zanten, J.~H. (2009).
\newblock Adaptive {B}ayesian estimation using a {G}aussian random field with
  inverse gamma bandwidth.
\newblock {\em Ann. Statist.}, 37(5B):2655--2675.

\bibitem[van Erven and Harremo\"es, 2014]{vanerven.renyi}
van Erven, T. and Harremo\"es, P. (2014).
\newblock R\'enyi divergence and {K}ullback-{L}eibler divergence.
\newblock {\em IEEE Trans. Inform. Theory}, 60(7):3797--3820.

\bibitem[Walker and Hjort, 2001]{walker.hjort.2001}
Walker, S. and Hjort, N.~L. (2001).
\newblock On {B}ayesian consistency.
\newblock {\em J. R. Stat. Soc. Ser. B Stat. Methodol.}, 63(4):811--821.

\bibitem[Walker et~al., 2007]{walker2007}
Walker, S.~G., Lijoi, A., and Pr{\"u}nster, I. (2007).
\newblock On rates of convergence for posterior distributions in
  infinite-dimensional models.
\newblock {\em Ann. Statist.}, 35(2):738--746.

\end{thebibliography}


\begin{thebibliography}{}

\bibitem[Arbel et~al., 2013]{arbel.etal.sjs2013}
Arbel, J., Gayraud, G., and Rousseau, J. (2013).
\newblock Bayesian optimal adaptive estimation using a sieve prior.
\newblock {\em Scand. J. Stat.}, 40(3):549--570.

\bibitem[Belitser, 2017]{belitser.ddm}
Belitser, E. (2017).
\newblock On coverage and local radial rates of credible sets.
\newblock {\em Ann. Statist.}, 45(3):1124--1151.

\bibitem[Bissiri et~al., 2016]{bissiri.holmes.walker.2016}
Bissiri, P.~G., Holmes, C.~C., and Walker, S.~G. (2016).
\newblock A general framework for updating belief distributions.
\newblock {\em J. R. Stat. Soc. Ser. B. Stat. Methodol.}, 78(5):1103--1130.

\bibitem[Castillo and van~der Vaart, 2012]{castillo.vaart.2012}
Castillo, I. and van~der Vaart, A. (2012).
\newblock Needles and straw in a haystack: posterior concentration for possibly
  sparse sequences.
\newblock {\em Ann. Statist.}, 40(4):2069--2101.

\bibitem[Donnet et~al., 2018]{rousseau.etal.eb}
Donnet, S., Rivoirard, V., Rousseau, J., and Scricciolo, C. (2018).
\newblock Posterior concentration rates for empirical {B}ayes procedures with
  applications to {D}irichlet process mixtures.
\newblock {\em Bernoulli}, 24(1):231--256.

\bibitem[Donoho et~al., 1992]{donoho1992}
Donoho, D.~L., Johnstone, I.~M., Hoch, J.~C., and Stern, A.~S. (1992).
\newblock Maximum entropy and the nearly black object.
\newblock {\em J. Roy. Statist. Soc. Ser. B}, 54(1):41--81.
\newblock With discussion and a reply by the authors.

\bibitem[Efron, 2014]{efron2014}
Efron, B. (2014).
\newblock Two modeling strategies for empirical {B}ayes estimation.
\newblock {\em Statist. Sci.}, 29(2):285--301.

\bibitem[Gao and Zhou, 2016]{gao.zhou.2016}
Gao, C. and Zhou, H.~H. (2016).
\newblock Rate exact {B}ayesian adaptation with modified block priors.
\newblock {\em Ann. Statist.}, 44(1):318--345.

\bibitem[Geman and Hwang, 1982]{geman1982}
Geman, S. and Hwang, C.-R. (1982).
\newblock Nonparametric maximum likelihood estimation by the method of sieves.
\newblock {\em Ann. Statist.}, 10(2):401--414.

\bibitem[Ghosal et~al., 2000]{ggv2000}
Ghosal, S., Ghosh, J.~K., and van~der Vaart, A.~W. (2000).
\newblock Convergence rates of posterior distributions.
\newblock {\em Ann. Statist.}, 28(2):500--531.

\bibitem[Ghosal and van~der Vaart, 2001]{ghosalvaart2001}
Ghosal, S. and van~der Vaart, A.~W. (2001).
\newblock Entropies and rates of convergence for maximum likelihood and {B}ayes
  estimation for mixtures of normal densities.
\newblock {\em Ann. Statist.}, 29(5):1233--1263.

\bibitem[Ghosal and van~der Vaart, 2007]{ghosalvaart2007}
Ghosal, S. and van~der Vaart, A.~W. (2007).
\newblock Posterior convergence rates of {D}irichlet mixtures at smooth
  densities.
\newblock {\em Ann. Statist.}, 35(2):697--723.

\bibitem[Grenander, 1981]{grenander1981}
Grenander, U. (1981).
\newblock {\em Abstract Inference}.
\newblock John Wiley \& Sons, Inc., New York.
\newblock Wiley Series in Probability and Mathematical Statistics.

\bibitem[Gr\"unwald and van Ommen, 2017]{grunwald.ommen.scaling}
Gr\"unwald, P. and van Ommen, T. (2017).
\newblock Inconsistency of {B}ayesian inference for misspecified linear models,
  and a proposal for repairing it.
\newblock {\em Bayesian Anal.}, 12(4):1069--1103.

\bibitem[Holmes and Walker, 2017]{holmes.walker.scaling}
Holmes, C.~C. and Walker, S.~G. (2017).
\newblock Assigning a value to a power likelihood in a general {B}ayesian
  model.
\newblock {\em Biometrika}, 104(2):497--503.

\bibitem[Kruijer et~al., 2010]{kruijer.rousseau.vaart.2010}
Kruijer, W., Rousseau, J., and van~der Vaart, A. (2010).
\newblock Adaptive {B}ayesian density estimation with location-scale mixtures.
\newblock {\em Electron. J. Stat.}, 4:1225--1257.

\bibitem[Lee et~al., 2017]{lee.lee.lin.deb}
Lee, K., Lee, J., and Lin, L. (2017).
\newblock Minimax posterior convergence rates and model selection consistency
  in high-dimensional {DAG} models based on sparse {C}holesky factors.
\newblock Unpublished manuscript, {\tt arXiv:...}

\bibitem[Martin, 2017a]{martin.horseshoe.discuss}
Martin, R. (2017a).
\newblock Comment on the article by van der {P}as, {S}zab\'o, and van der
  {V}aart.
\newblock {\em Bayesian Anal.}, 12(4):1254--1258.

\bibitem[Martin, 2017b]{ebmono}
Martin, R. (2017b).
\newblock Empirical priors and posterior concentration rates for a monotone
  density.
\newblock Unpublished manuscript, {\tt arXiv:1706.08567}.

\bibitem[Martin et~al., 2017]{martin.mess.walker.eb}
Martin, R., Mess, R., and Walker, S.~G. (2017).
\newblock Empirical {B}ayes posterior concentration in sparse high-dimensional
  linear models.
\newblock {\em Bernoulli}, 23(3):1822--1847.

\bibitem[Martin and Shen, 2017]{ebpiece}
Martin, R. and Shen, W. (2017).
\newblock Asymptotically optimal empirical {B}ayes inference in a piecewise
  constant sequence model.
\newblock Unpublished manuscript, {\tt arXiv:1712.03848}.

\bibitem[Martin and Walker, 2014]{martin.walker.eb}
Martin, R. and Walker, S.~G. (2014).
\newblock Asymptotically minimax empirical {B}ayes estimation of a sparse
  normal mean vector.
\newblock {\em Electron. J. Stat.}, 8(2):2188--2206.

\bibitem[Miller and Dunson, 2015]{miller.dunson.power}
Miller, J.~W. and Dunson, D.~B. (2015).
\newblock Robust {B}ayesian inference via coarsening.
\newblock Unpublished manuscript, {\tt arXiv:1506.06101}.

\bibitem[Rousseau and Szabo, 2017]{rousseau.szabo.2017}
Rousseau, J. and Szabo, B. (2017).
\newblock Asymptotic behaviour of the empirical {B}ayes posteriors associated
  to maximum marginal likelihood estimator.
\newblock {\em Ann. Statist.}, 45(2):833--865.

\bibitem[Salomond, 2014]{salomond2014}
Salomond, J.-B. (2014).
\newblock Concentration rate and consistency of the posterior distribution for
  selected priors under monotonicity constraints.
\newblock {\em Electron. J. Stat.}, 8(1):1380--1404.

\bibitem[Scricciolo, 2007]{scricciolo2007}
Scricciolo, C. (2007).
\newblock On rates of convergence for {B}ayesian density estimation.
\newblock {\em Scand. J. Statist.}, 34(3):626--642.

\bibitem[Scricciolo, 2015]{scricciolo2015}
Scricciolo, C. (2015).
\newblock Bayesian adaptation.
\newblock {\em J. Statist. Plann. Inference}, 166:87--101.

\bibitem[Shen and Ghosal, 2015]{shen.ghosal.2015}
Shen, W. and Ghosal, S. (2015).
\newblock Adaptive {B}ayesian procedures using random series priors.
\newblock {\em Scand. J. Stat.}, 42(4):1194--1213.

\bibitem[Shen and Wasserman, 2001]{shen.wasserman.2001}
Shen, X. and Wasserman, L. (2001).
\newblock Rates of convergence of posterior distributions.
\newblock {\em Ann. Statist.}, 29(3):687--714.

\bibitem[Syring and Martin, 2016]{syring.martin.scaling}
Syring, N. and Martin, R. (2016).
\newblock Calibrating general posterior credible regions.
\newblock Unpublished manuscript, {\tt arXiv:1509.00922}.

\bibitem[van~der Vaart and van Zanten, 2009]{vaart.zanten.2009}
van~der Vaart, A.~W. and van Zanten, J.~H. (2009).
\newblock Adaptive {B}ayesian estimation using a {G}aussian random field with
  inverse gamma bandwidth.
\newblock {\em Ann. Statist.}, 37(5B):2655--2675.

\bibitem[Walker and Hjort, 2001]{walker.hjort.2001}
Walker, S. and Hjort, N.~L. (2001).
\newblock On {B}ayesian consistency.
\newblock {\em J. R. Stat. Soc. Ser. B Stat. Methodol.}, 63(4):811--821.

\bibitem[Walker et~al., 2007]{walker2007}
Walker, S.~G., Lijoi, A., and Pr{\"u}nster, I. (2007).
\newblock On rates of convergence for posterior distributions in
  infinite-dimensional models.
\newblock {\em Ann. Statist.}, 35(2):738--746.

\end{thebibliography}
%\bibliography{mybib}
}{

}

\end{document}